\documentclass{article}

\usepackage{amsmath,amsthm,amssymb}
\usepackage{hyperref}
\usepackage{mathtools}
\usepackage{pgfplots}
\pgfplotsset{compat=1.17}
\usepackage{subcaption}
\usepackage{tikz}
\usetikzlibrary{angles,quotes}

\newtheorem{theorem}{Theorem}
\newtheorem{lemma}[theorem]{Lemma}
\newtheorem{remark}[theorem]{Remark}
\newtheorem{example}[theorem]{Example}

\newtheorem{proposition}[theorem]{Proposition}
\newtheorem{corollary}[theorem]{Corollary}

\newcommand{\calB}{\mathcal{B}}

\newcommand{\calF}{\mathcal{F}}
\newcommand{\calG}{\mathcal{G}}

\newcommand{\calO}{\mathcal{O}}

\newcommand{\bbC}{\mathbb{C}}

\newcommand{\bbN}{\mathbb{N}}

\newcommand{\bbQ}{\mathbb{Q}}
\newcommand{\bbR}{\mathbb{R}}
\newcommand{\bbZ}{\mathbb{Z}}

\newcommand{\rme}{\mathrm{e}}
\newcommand{\rmi}{\mathrm{i}}

\newcommand{\set}[2]{\left\{ #1 \,\middle|\, #2 \right\}}
\newcommand{\abs}[1]{\left\lvert #1 \right\rvert}
\newcommand{\norm}[1]{\left\lVert #1 \right\rVert}

\DeclareMathOperator{\supp}{supp}

\DeclareMathOperator{\csch}{csch}

\renewcommand{\Re}{\operatorname{Re}}
\renewcommand{\Im}{\operatorname{Im}}

\newcommand{\dd}{\,\mathrm{d}}

\title{Phase retrieval of entire functions and its implications for Gabor phase retrieval}
\author{Matthias Wellershoff\thanks{University of Maryland, Department of Mathematics, 4176 Campus Drive, College Park, MD 20742, USA, \href{mailto:wellersm@umd.edu}{wellersm@umd.edu}}}
\date{\today}

\begin{document}
\maketitle

\begin{abstract}
    We characterise all pairs of finite order entire functions whose magnitudes agree on two arbitrary lines in the complex plane by means of the Hadamard factorisation theorem. Building on this, we also characterise all pairs of second order entire functions whose magnitudes agree on infinitely many equidistant parallel lines. Furthermore, we show that the magnitude of an entire function on three parallel lines, whose distances are rationally independent, uniquely determines the function up to global phase, and that there exists a first order entire function whose magnitude on the lines $\bbR + \tfrac{\rmi}{n} \bbZ$ does \emph{not} uniquely determine it up to global phase, for all positive integers $n$. Our results have direct implications for Gabor phase retrieval.

    \vspace{5pt}

    \vspace{5pt}
\end{abstract} 

\section{Introduction}

Let us consider the recovery of entire functions $f$ from partial knowledge of their magnitudes
\begin{equation*}
    \lvert f (z) \rvert, \qquad z \in \Omega \subseteq \bbC.
\end{equation*}
The above problem is called \emph{phase retrieval} of entire functions. Clearly, $f$ cannot be recovered completely since the entire function $\rme^{\rmi \alpha} f$, with $\alpha \in \bbR$, has the same magnitudes. To capture this, we will consider the equivalence relation $\sim$ defined by 
\begin{equation}\label{eq:globalphase}
    f \sim g :\iff \exists \alpha \in \bbR. f = \rme^{\rmi \alpha} g.
\end{equation}
Recovering an equivalence class under $\sim$ is referred to as recovery \emph{up to global phase}.

The interest in the phase retrieval problem of entire functions arises from its direct connection to various other phase retrieval problems encountered in applications. A notable example is \emph{Gabor phase retrieval}, where the objective is to reconstruct a square-integrable function (up to global phase) from the magnitude of its Gabor transform. This particular problem finds applications in audio processing: in particular, the Griffin--Lim algorithm \cite{griffin1984signal}, designed for Gabor phase retrieval, is integrated into PyTorch's torchaudio library and TensorFlow's I/O module.

The main result of this paper is the \emph{characterisation all finite order entire functions which agree on two arbitrary lines} in the complex plane (cf.~Theorems~\ref{thm:irrational_intersection}, \ref{thm:rational_intersection}, and \ref{thm:parallel}), which is proven using the well-known Hadamard factorisation theorem. This characterisation includes a theorem proven in \cite{jaming2014uniqueness} addressing the unique recovery (up to global phase) of finite order entire functions from magnitude information on two intersecting lines.

Additionally, we consider entire functions whose magnitudes agree on three or more parallel lines in the complex plane. This leads us to a new uniqueness result on lines whose distances are rationally independent.

\begin{theorem}
    Let $\ell_1,\ell_2,\ell_3 \subset \bbC$ be three parallel lines. Let $a > 0$ denote the distance between $\ell_1$ and $\ell_2$ and let $b > 0$ denote the distance between $\ell_2$ and $\ell_3$. Assume that $\tfrac{b}{a}$ is irrational and let $f$ as well as $g$ be entire functions. Then, $\lvert f \rvert = \lvert g \rvert$ holds on $\ell_1 \cup \ell_2 \cup \ell_3$ if and only if $f \sim g$.
\end{theorem}

Moreover, our characterisation allows for the construction of entire functions which do not agree up to global phase but whose magnitudes agree on infinitely many parallel lines; such as, 
\begin{equation}\label{eq:counterexamples}
    f(z) = \cosh z + \rmi \sinh z, \qquad g(z) = \cosh z - \rmi \sinh z.
\end{equation}
These functions have matching magnitudes on $\bbR + \tfrac{\pi \rmi}{2} \bbZ$ and share no common zeroes. This construction was used to generate counterexamples to uniqueness for sampled Gabor phase retrieval in \cite{alaifari2021phase} (which was later generalised in \cite{alaifari2024connection,grohs2022foundational}). 

Finally, our characterisation leads to the construction of an entire function $f_0$ for which there exist parallel lines that are arbitrarily close together on which magnitude information of $f_0$ does not suffice for a recovery up to global phase. 

\begin{theorem}
    There exists a sequence of entire functions $(f_n)_{n \in \bbN_0}$ of order one such that $f_0 \not\sim f_n$ and $\lvert f_0 \rvert = \lvert f_n \rvert$ on $\bbR + \tfrac{\rmi}{n} \bbZ$ for all $n \in \bbN$. 
\end{theorem}

\section{Notation, preliminaries and overview of existing literature}

\subsection{Notation}

The space of entire functions is denoted by $\mathcal{O}(\mathbb{C})$. Complex polynomials of degree at most $d \in \mathbb{N}_0$ are denoted by $\mathbb{C}_d[z]$, and the space of all complex polynomials is $\mathbb{C}[z]$. The equivalence relation $\sim$ defined in \eqref{eq:globalphase} is primarily associated with functions in $\mathcal{O}(\mathbb{C})$ in this paper, though occasionally also used for functions in $L^2(\mathbb{R})$.

An entire function $f$ is of \emph{finite order} if it satisfies $f(z) = \mathcal{O}(\exp(\lvert z \rvert^a))$ as $z \to \infty$ for some positive $a > 0$. The infimum of all such $a$ is called the \emph{order} $\rho \geq 0$ of the function $f$. Similarly, an entire function $f$ of order $\rho > 0$ is of \emph{type} $\sigma \geq 0$ if $\sigma$ is the infimum of all positive $k > 0$ such that $f(z) = \mathcal{O}(\exp(k \lvert z \rvert^\rho))$ as $z \to \infty$.

The zeroes of an entire function $f$ are described using the multiset notation $M_f : \bbC \to \bbN_0$, 
\begin{equation*}
    M_f(z) := \begin{cases}
        m & \mbox{if } z \mbox{ is a zero of } f \mbox{ of multiplicity } m, \\
        0 & \mbox{else}.
    \end{cases}
\end{equation*}
The expressions
\begin{equation*}
    E(z;0) = 1 - z, \qquad E(z;p) = (1-z) \exp \Bigg( \sum_{j = 1}^p z^j/j \Bigg),
\end{equation*}
for $p \in \bbN$, are called \emph{primary factors}. If $f$ is a finite order entire function, then there exists an integer $p \in \bbN$ such that (see e.g.~\cite[Subsection~8.23 on p.~250]{titchmarsh1939theory})
\begin{equation*}
    \sum_{a \in \supp(M_f) \setminus \{0\}} \frac{M_f(a)}{\lvert a \rvert^{p+1}} < \infty.
\end{equation*}
If $p$ is the smallest such integer, we call 
\begin{equation*}
    P(z) := z^{M_f(0)} \prod_{a \in \supp(M_f) \setminus \{0\}} E\Big(\frac{z}{a};p\Big)^{M_f(a)}
\end{equation*}
the \emph{canonical product} formed with the zeroes of $f$, and we call $p$ its \emph{genus}. Here and throughout the rest of this paper, $\supp$ denotes the support of a function. Finally, we denote the complex phase or argument by $\arg : \bbC \to (-\pi,\pi]$.

\subsection{Preliminaries}

It is not difficult to see (and widely known) that phase retrieval of entire functions enjoys uniqueness when magnitude information on an open set is available.

\begin{lemma}\label{lem:folklore}
    Let $\Omega \subseteq \bbC$ be an open set, and let $f,g \in \calO(\bbC)$ be entire functions. Then, $\lvert f \rvert = \lvert g \rvert$ on $\Omega$ if and only if $f \sim g$.
\end{lemma}

We include a very short proof for the convenience of the reader.

\begin{proof}[Proof of Lemma~\ref{lem:folklore}]
    The lemma follows from the identity theorem if $g = 0$. So, let us assume that $g \neq 0$ and $\lvert f \rvert = \lvert g \rvert$ on $\Omega$. Let $Z(g) \subset \bbC$ denote the set of zeroes of $g$. By assumption, $\lvert f/g \rvert = 1$ on $\Omega \setminus Z(g)$ such that the open mapping theorem implies that $f/g$ is constant on $\Omega \setminus Z(g)$. Hence, $f/g = \rme^{\rmi \alpha}$ on $\Omega \setminus Z(g)$ for some $\alpha \in \bbR$. Now, the identity theorem implies that $f \sim g$. 
\end{proof}

Therefore, phase retrieval of entire functions is unique when magnitude information on a set \emph{containing an open set} is available. It remains to consider sets which do \emph{not} contain open sets. We start by considering individual straight lines.

Let $f,g \in \calO(\bbC)$ be two non-zero entire functions and split the multiset of zeroes of $f$ into two smaller multisets: the common zeroes of $f$ and $g$, and the zeroes of $f$ that are no zeroes of $g$. The former are described by the multiset $M_\text{c} := \min\{ M_f, M_g \}$ while the latter are described by the multiset $M_\text{d} := \min\{ M_f - M_g,0 \}$.

Using this notation, we present the following theorem from \cite{mcdonald2004phase} which will be applied repeatedly in the rest of this paper. An earlier reference, \cite{walther1963question}, establishes the result for Fourier transforms of functions in $L^2([-1/2,1/2])$. The proof notably also covers the more general case presented below. Additionally, in \cite{akutowicz1957determination}, the theorem is proven using Blaschke-type products instead of canonical products.

\begin{theorem}\label{thm:zeroflipping}
    Let $f,g \in \calO(\bbC) \setminus \{0\}$ be of finite order. Then, $\abs{f} = \abs{g}$ on $\bbR$ if and only if $f$ and $g$ factor as 
    \begin{gather*}
        f(z) = \rme^{Q_f(z)} \cdot z^m \prod_{a \in \operatorname{supp}(M_\text{c}) \setminus \{0\}} E\left( \frac{z}{a} ; p \right)^{M_\text{c}(a)} \prod_{a \in \operatorname{supp}(M_\text{d})} E\left( \frac{z}{a} ; p \right)^{M_\text{d}(a)}, \\
        g(z) = \rme^{Q_g(z)} \cdot z^m \prod_{a \in \operatorname{supp}(M_\text{c}) \setminus \{0\}} E\left( \frac{z}{a} ; p \right)^{M_\text{c}(a)} \prod_{a \in \operatorname{supp}(M_\text{d})} E\left( \frac{z}{\overline a} ; p \right)^{M_\text{d}(a)}.
    \end{gather*}
    Here, $p \in \bbN_0$ is the genus of the canonical product formed with the zeroes of $f$, $m \in \bbN_0$ is the multiplicity of the zero at the origin of $f$, and $Q_f, Q_g \in \bbC[z]$ are polynomials of the same order which have coefficients whose real parts agree.
\end{theorem}

The reader is referred to \cite{mcdonald2004phase,walther1963question} for the proof of the theorem above, which is based on the well-known Hadamard factorisation theorem (cf.~e.g.~\cite[Subsection~8.24 on p.~250]{titchmarsh1939theory}). For the remainder of this paper, the following core idea is especially important: if $\abs{f} = \abs{g}$ on $\bbR$, then
\begin{equation*}
    f(z) \overline{ f (\overline z) } = g(z) \overline{ g (\overline z) }, \qquad z \in \bbC,
\end{equation*}
by the identity theorem of complex analysis. Therefore, the multisets $M_f$ and $M_g$ satisfy
\begin{equation}
    \label{eq:mirror_symmetry}
    M_f(z) + M_f(\overline z) = M_g(z) + M_g(\overline z), \qquad z \in \bbC.
\end{equation}

\begin{remark}\label{rem:symmetries}
    Let us make three further remarks:
    \begin{enumerate}
        \item Equation~\eqref{eq:mirror_symmetry} immediately shows that $\abs{M_f - M_g}$, the symmetric difference of the multisets of the zeroes of $f$ and $g$, is invariant under the action of complex conjugation. 
        \item The theorem generalises to entire functions $f$ and $g$ whose magnitudes agree on a general line $\ell$ in $\bbC$ by composing $f$ and $g$ with a rotation and a translation. In this case, if $\sigma_\ell : \bbC \to \bbC$ denotes the reflection of $\bbC$ along $\ell$, then 
        \begin{equation}
            \label{eq:mirror_symmetry_general}
            M_f(z) + M_f(\sigma_\ell(z)) = M_g(z) + M_g(\sigma_\ell(z)), \qquad z \in \bbC,
        \end{equation}
        and $\abs{M_f - M_g}$ is invariant under the action of $\sigma_\ell$.
        \item In the context of Theorem~\ref{thm:zeroflipping}, we define the function
        \begin{equation*}
            P_\mathrm{c}(z) := z^m \prod_{a \in \operatorname{supp}(M_\text{c}) \setminus \{0\}} E\left( \frac{z}{a} ; p \right)^{M_\text{c}(a)}
        \end{equation*}
        and call it the \emph{canonical product formed with the common zeroes of $f$ and $g$}.
    \end{enumerate}
\end{remark}

\subsection{Background and overview of existing literature}

The results presented in this paper were established to answer open questions in the Gabor phase retrieval community. Following the findings in \cite{alaifari2021uniqueness}, which demonstrated that real-valued bandlimited functions are uniquely determined by samples of their Gabor transform magnitude, there was interest in exploring whether the same holds for general complex-valued (bandlimited) functions. While this question was later resolved in \cite{grohs2023injectivity} and \cite{wellershoff2024injectivity} using different techniques, at the time, there was hope that the connection between the Gabor and Bargmann transforms could be used in this pursuit. Since the range of the Bargmann transform is the Fock space of second order entire functions \cite[Section~3.4 on pp.~53--58]{groechenig2001foundations}, one is naturally led to questions about phase retrieval of entire functions.

The approach adopted here draws inspiration from \cite{mallat2015phase}, where one of the key elements is a result establishing the uniqueness of recovering certain functions holomorphic on the upper half-plane from magnitude measurements on two parallel lines. This result is then applied to show that measurements of the Cauchy wavelet transform on two scales are sufficient for the recovery of the analytic representations\footnote{The \emph{analytic representation} $f_+$ of $f \in L^2(\bbR)$ is defined by $\widehat f_+ := 2\chi_{\bbR_+} \cdot \widehat f$, where $\widehat \cdot$ denotes the Fourier transform and $\chi_{\bbR_+}$ is the indicator function of the positive real numbers $\bbR_+$.} of square-integrable functions (up to global phase). Consequently, the analytic representations of bandlimited functions can be reconstructed from samples of the Cauchy wavelet transform \cite{alaifari2023phase}. These ideas motivated the related problem addressed in this paper: the recovery of second order entire functions from their magnitudes on two parallel lines.

Finally, we note that techniques analogous to those employed in this paper were used in \cite{groechenig2020phase} to address the problem of phase retrieval in shift-invariant spaces with Gaussian generator. Similarly, in \cite{wellershoff2022sampling}, comparable methods were applied to demonstrate the unique recovery (up to global phase) of bandlimited functions from Gabor transform magnitudes sampled at twice the Nyquist rate in two frequency bins.

\section{Characterisation of finite order entire functions with common magnitudes on two lines}

Theorem~\ref{thm:zeroflipping} states that the magnitude of an entire function $f$ on a single line does not uniquely determine $f$ (up to global phase). Two ambiguities arise: first, magnitude information on the real line $\mathbb{R}$ does not suffice to distinguish between functions of the form $\exp(i Q)$, where $Q \in \mathbb{R}[z]$ is a polynomial with real coefficients; secondly, and potentially more importantly, magnitude information on any line is insufficient to uniquely determine zeroes of entire functions.

In this section, we investigate whether magnitude information on \emph{two} lines uniquely determines entire functions (up to global phase). The insights presented in the following can be generalised to cases in which we have magnitude information on an arbitrary number of lines under certain conditions.

\subsection{Intersecting lines}

Let us consider two arbitrary lines in the complex plane. There are two possible scenarios: either the lines intersect at some point, or they are parallel. In this subsection, we focus on intersecting lines, and assume, without loss of generality\footnote{Otherwise, we can translate and rotate the lines.}, that they meet at the origin, one of the lines is $\mathbb{R}$ and the other intersects $\bbR$ at an angle $\theta \in (0,\tfrac{\pi}{2}]$. We further categorise our analysis into two cases: first, when $\theta \in (0,\tfrac{\pi}{2}] \setminus \pi \mathbb{Q}$ (we call such angles \emph{irrational}), and secondly, when $\theta \in (0,\tfrac{\pi}{2}] \cap \pi \mathbb{Q}$ (we call such angles \emph{rational}).

Let us begin by considering a general angle $\theta \in (0, \tfrac{\pi}{2}]$, however. For two non-zero entire functions $f, g \in \mathcal{O}(\mathbb{C})$ with $\lvert f \rvert = \lvert g \rvert$ on $\mathbb{R} \cup \mathrm{e}^{\mathrm{i}\theta} \mathbb{R} \subset \mathbb{C}$, the function $\lvert M_f - M_g \rvert$ is invariant under both the reflection $\sigma_\theta$ along $\mathrm{e}^{\mathrm{i}\theta} \mathbb{R}$ and complex conjugation (which we may denote by $\sigma_0$ since it coincides with reflection along $\bbR$) according to Remark~\ref{rem:symmetries}. Consequently, $\lvert M_f - M_g \rvert$ is invariant under the action of the group $G$ generated by $\sigma_0$ and $\sigma_\theta$. Notably, this group contains the rotation $\rho_{2\theta} = \sigma_0 \sigma_\theta$ by angle $2\theta$.

\subsubsection*{Irrational angles}

If $\theta \in (0,\tfrac{\pi}{2}] \setminus \pi \bbQ$ is an irrational angle, then the invariance of the function $\lvert M_f - M_g \rvert$ under the rotation $\rho_{2\theta}$ allows us to conclude that $M_f = M_g$. Indeed, remember that equation~\eqref{eq:mirror_symmetry} implies that $M_f(0) = M_g(0)$, and suppose by contradiction that there exists a number $a \in \bbC \setminus \{0\}$ such that $\lvert M_f(a) - M_g(a) \rvert > 0$. Then, $\lvert M_f - M_g \rvert$ is non-zero on the countably infinite set $\rho_{2\theta}^{\bbZ} a$ contained in the circle of radius $\lvert a \rvert$ around the origin. Therefore, either $f$ or $g$ must have a non-isolated set of zeroes and cannot be non-zero.

Hence, magnitude information on two lines intersecting at an irrational angle uniquely determines the zeroes of an entire function. One can now show that this enables us to completely recover the entire function up to global phase (which was first proven in \cite{jaming2014uniqueness}).

\begin{theorem}[{\cite[Theorem~3.3 on p.~419]{jaming2014uniqueness}}]\label{thm:irrational_intersection}
    Let $\ell_1, \ell_2 \subset \bbC$ be two lines intersecting at an irrational angle $\theta \in \bbR \setminus \pi \bbQ$, and let $f,g \in \calO(\bbC)$ be two entire functions of finite order. Then, $\lvert f \rvert = \lvert g \rvert$ on $\ell_1 \cup \ell_2$ if and only if $f \sim g$.
\end{theorem}

The proof of the above theorem can be found in \cite{jaming2014uniqueness} which also contains the insight that it actually suffices to assume that $\lvert f \rvert = \lvert g \rvert$ on $u_1 \cup u_2$, where $u_1 \subseteq \ell_1$ and $u_2 \subseteq \ell_2$ are sets of uniqueness for finite order entire functions.

\begin{remark}\label{rem:reflection_principle}
    Theorem~\ref{thm:irrational_intersection} remains true for general entire functions. Indeed, the argument that shows that the zeroes of $f$ and $g$ agree remains true. Then, Weierstrass factorisation implies that $f = \rme^{\alpha} P$ and $g = \rme^{\beta} P$, where $P$ is a product of primary factors and $\alpha,\beta \in \calO(\bbC)$. Since $\lvert f \rvert = \lvert g \rvert$ on $\ell_1 \cup \ell_2$, we have that $u := \Re \alpha$ and $v := \Re \beta$ satisfy $u - v = 0$ on $\ell_1 \cup \ell_2$. Now, $u-v$ is harmonic and the reflection principle of harmonic functions (see e.g.~\cite[Theorem~24 on pp.~172--173]{ahlfors1979complex}) implies that $u-v = 0$ on $\rho_{2 \theta}^\bbZ \ell_1$. Since $\rho_{2 \theta}^\bbZ \ell_1 \subset \bbC \cong \bbR^2$ is dense, it follows that $u - v = 0$ and thus that $\alpha - \beta = \rmi c$ for some $c \in \bbR$. 
\end{remark}

\subsubsection*{Rational angles}

If $\theta \in (0,\tfrac{\pi}{2}] \setminus \pi \bbQ$ is a rational angle, then we cannot conclude that $f$ and $g$ have the same zeroes. Consequently, a finite order entire function is \emph{not} uniquely determined by its magnitude on two lines intersecting at angle $\theta$. A simple example can be found when $\theta = \tfrac{\pi}{2}$: the functions
\begin{equation*}
    f(z) = 1 - \frac{z^2}{(1 + \rmi)^2}, \qquad g(z) = 1 - \frac{z^2}{(1 - \rmi)^2},
\end{equation*}
agree on $\bbR \cup \rmi \bbR$ and have distinct sets of zeroes.

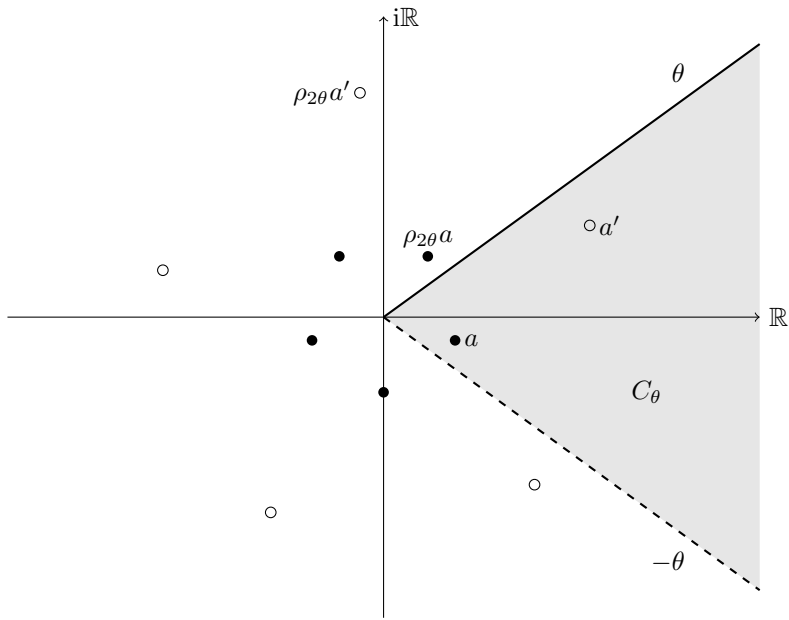
\begin{figure}
    \centering
    \begin{tikzpicture}

        \fill[gray!20] (0,0)--(5,3.633)--(5,-3.633);
        \node at (3.5,-1) {$C_\theta$};

        \draw[->] (-5,0)--(5,0) node[right] {$\bbR$};
        \draw[->] (0,-4)--(0,4) node[right] {$\rmi \bbR$};

        \draw[thick, dashed] (0,0)--(5,-3.633);
        \draw[thick] (0,0)--(5,3.633);
        \node[above left] at (4.129,3) {$\theta$};
        \node[below left] at (4.129,-3) {$-\theta$};

        \fill[black] (0.951,-0.309) circle (2pt) node[right] {$a$};
        \fill[black] (0.588,0.809) circle (2pt) node[above] {$\rho_{2 \theta} a$};
        \fill[black] (-0.588,0.809) circle (2pt);
        \fill[black] (-0.951,-0.309) circle (2pt);
        \fill[black] (0,-1) circle (2pt);

        \draw (2.741,1.22) circle (2pt) node[right] {$a'$};
        \draw (-0.314,2.984) circle (2pt) node[left] {$\rho_{2 \theta} a'$};
        \draw (-2.934,0.624) circle (2pt);
        \draw (-1.5,-2.598) circle (2pt);
        \draw (2.007,-2.229) circle (2pt);

    \end{tikzpicture}
    \caption{The grouping of zeroes in $\supp(M_\mathrm{d})$ using the rotational symmetry with angle $2 \theta$. The dotted line indicates the open boundary of $C_\theta$. Note that, for all $a,a' \in \supp(M_\mathrm{d})$, it holds that $\rho_{2 \theta}^{k} a, \rho_{2 \theta}^{k} a' \in \supp(M_\mathrm{d})$, for $k=1,\dots,n$.}
    \label{fig:intersecting_splitting_zeroes}
\end{figure}

In the following, we classify all finite order entire functions $f,g \in \calO(\bbC)$ whose magnitudes agree on two lines intersecting at a rational angle. At the core of this classification is a simple grouping of the zeroes of $f$. Recall that we denoted the multiset of common zeroes of $f$ and $g$ by $M_\mathrm{c} := \min\{ M_f , M_g \}$ while the multiset of zeroes of $f$ that are no zeroes of $g$ was denoted by $M_\mathrm{d} := \min\{ M_f - M_g , 0 \}$.

We can further split the second multiset because of its rotational symmetry: by equations~\eqref{eq:mirror_symmetry} and \eqref{eq:mirror_symmetry_general}, we have $M_\mathrm{d} = M_\mathrm{d} \circ \rho_{2\theta}^k$ for all $k \in \bbZ$. Consequently, the values of $M_\mathrm{d}$ are completely determined by its values on the (infinitely extended) cone (cf.~Figure~\ref{fig:intersecting_splitting_zeroes})
\begin{equation*}
    C_\theta := \{ z \in \bbC \,|\, \arg z \in (-\theta,\theta] \}.
\end{equation*}
In particular, if $n \in \bbN$ is the smallest integer such that $\theta n \in \pi \bbN$, then
\[
    \supp(M_\mathrm{d}) = \rho_{2 \theta}^{\bbZ/n\bbZ} (\supp(M_\mathrm{d}) \cap C_\theta).
\]
We can thereby establish the following theorem.

\begin{theorem}\label{thm:rational_intersection}
    Let $\theta \in \pi \bbQ \cap (0,\frac{\pi}{2}]$, let $n \in \bbN$ be the smallest integer such that $\theta n \in \pi \bbN$ and let $f,g \in \calO(\bbC) \setminus \{0\}$ be of finite order. Then, $\abs{f} = \abs{g}$ on $\bbR \cup \rme^{\rmi \theta} \bbR$ if and only if $f$ and $g$ factor as
    \begin{gather*}
        f(z) = \rme^{Q_f(z)} P_\mathrm{c}(z) \prod_{\substack{a \in \operatorname{supp}(M_\text{d}) \\ \arg a \in (-\theta,\theta]}} \prod_{k = 1}^n E\left( \frac{z}{\rho_{2 \theta}^k a} ; p \right)^{M_\text{d}(a)}, \\
        g(z) = \rme^{Q_g(z)} P_\mathrm{c}(z) \prod_{\substack{a \in \operatorname{supp}(M_\text{d}) \\ \arg a \in (-\theta,\theta]}} \prod_{k = 1}^n E\left( \frac{z}{\rho_{2 \theta}^k \overline a} ; p \right)^{M_\text{d}(a)}.
    \end{gather*}
    Here, $p \in \bbN_0$ is the genus of the canonical product formed with the zeroes of $f$, $P_\mathrm{c}$ is the canonical product formed with the common zeroes of $f$ and $g$, and $Q_f, Q_g \in \bbC[z]$ are polynomials of the same order $d \in \bbN_0$, whose coefficients have the same real parts and whose imaginary parts $(\lambda_k)_{k = 0}^d, (\mu_k)_{k = 0}^d \in \bbR$ satisfy $\lambda_k = \mu_k$ for $k \not\in n \bbZ$.
\end{theorem}

\begin{proof}
    Let us show that $f$ and $g$ satisfying $\abs{f} = \abs{g}$ on $\bbR \cup \rme^{\rmi \theta} \bbR$ can be factored as described above. We note that Theorem~\ref{thm:zeroflipping} along with the considerations before the statement of this theorem show that 
    \begin{gather*}
        f(z) = \rme^{Q_f(z)} P_\mathrm{c}(z) \prod_{\substack{a \in \operatorname{supp}(M_\text{d}) \\ \arg a \in (-\theta,\theta]}} \prod_{k = 1}^n E\left( \frac{z}{\rho_{2 \theta}^k a} ; p \right)^{M_\text{d}(a)}, \\
        g(z) = \rme^{Q_g(z)} P_\mathrm{c}(z) \prod_{\substack{a \in \operatorname{supp}(M_\text{d}) \\ \arg a \in (-\theta,\theta]}} \prod_{k = 1}^n E\left( \frac{z}{\rho_{2 \theta}^k \overline a} ; p \right)^{M_\text{d}(a)},
    \end{gather*}
    where $p \in \bbN_0$ is the genus of the canonical product formed with the zeroes of $f$, $P_\mathrm{c}$ is the canonical product formed with the common zeroes of $f$ and $g$, and $Q_f, Q_g \in \bbC[z]$ are polynomials of the same order whose coefficients $(\lambda_k)_{k = 0}^d, (\mu_k)_{k = 0}^d \in \bbC$ have the same real parts. Here, we use that $p$ is large enough to make the products converge unconditionally.

    By rearranging the product over $k$ in the above formula for $g$, we can see that 
    \begin{align*}
        \frac{\lvert f(\rme^{\rmi \theta} x) \rvert}{\lvert g(\rme^{\rmi \theta} x) \rvert} &= \rme^{\Re[ Q_f(\rme^{\rmi \theta} x) - Q_g(\rme^{\rmi \theta} x) ]} \prod_{\substack{a \in \operatorname{supp}(M_\text{d}) \\ \arg a \in (-\theta,\theta]}} \prod_{k = 1}^n \Bigg\lvert \frac{E\left( \frac{\rme^{\rmi \theta} x}{\rho_{2 \theta}^k a} ; p \right)}{E\left( \frac{\rme^{\rmi \theta} x}{\rho_{2 \theta}^{1-k} \overline a} ; p \right)} \Bigg\rvert^{M_\text{d}(a)} \\
        &= \rme^{\Re[ Q_f(\rme^{\rmi \theta} x) - Q_g(\rme^{\rmi \theta} x) ]} \prod_{\substack{a \in \operatorname{supp}(M_\text{d}) \\ \arg a \in (-\theta,\theta]}} \prod_{k = 1}^n \Bigg\lvert \frac{E( \rme^{(1-2k)\rmi\theta} a^{-1} x ; p )}{E( \rme^{(2k-1)\rmi\theta} \overline a^{-1} x ; p )} \Bigg\rvert^{M_\text{d}(a)} \\
        &= \rme^{\Re[ Q_f(\rme^{\rmi \theta} x) - Q_g(\rme^{\rmi \theta} x) ]},
    \end{align*}
    where we used $\lvert E(\overline z; p) \rvert = \lvert E(z; p) \rvert$. It follows that 
    \begin{equation*}
        \Im \left[ \sum_{k = 0}^d (\lambda_k - \mu_k) \rme^{\rmi k \theta} x^k \right] = 0
    \end{equation*}
    which immediately implies $(\lambda_k - \mu_k) \sin (k \theta) = 0$ for $k = 0,\dots,d$. Therefore, $\lambda_k = \mu_k$ for $k \not\in n \bbZ$. The reverse implication follows from the same arguments.
\end{proof}

\begin{remark}
    Using rigid motions, we can extend Theorem~\ref{thm:rational_intersection} to arbitrary lines $\ell_1$ and $\ell_2$ intersecting at a rational angle. From there on, we can further extend to sets of uniqueness $u_1 \subseteq \ell_1$ and $u_2 \subseteq \ell_2$ for finite order entire functions.
\end{remark}

\subsection{Parallel lines}

In the following, we will analyse entire functions whose magnitudes agree on two parallel lines in the complex plane. We assume, without loss of generality, that one of these parallel lines coincides with the real numbers while the other is $\mathbb{R} + \mathrm{i} y_0$ where $y_0 > 0$.

Let us start by noting that the magnitude of an entire function, measured on two parallel lines, does \emph{not} uniquely determine the zeroes of the entire function. Consider the functions $f,g \in \calO(\bbC)$ given by equation~\eqref{eq:counterexamples}. Then, it is readily seen that $f$ and $g$ have no common zeroes while satisfying $\lvert f \rvert = \lvert g \rvert$ on $\bbR$. Additionally, we can use trigonometric identities to calculate 
\begin{align*}
    \abs{f\left(x + \frac{\pi \rmi}{2}\right)} &= \abs{\cosh\left( x + \frac{\pi \rmi}{2} \right) + \rmi \sinh\left( x + \frac{\pi \rmi}{2} \right)} \\
        &= \abs{ \cosh x - \rmi \sinh x }
        = \abs{ \cosh x + \rmi \sinh x } \\
        &= \abs{\cosh\left( x + \frac{\pi \rmi}{2} \right) - \rmi \sinh\left( x + \frac{\pi \rmi}{2} \right)}
        = \abs{g\left(x + \frac{\pi \rmi}{2}\right)},
\end{align*}
for $x \in \bbR$. Of course, the above example can be adapted to any two parallel lines $\ell_1, \ell_2 \subset \bbC$ by appropriately translating, rotating and scaling the functions.

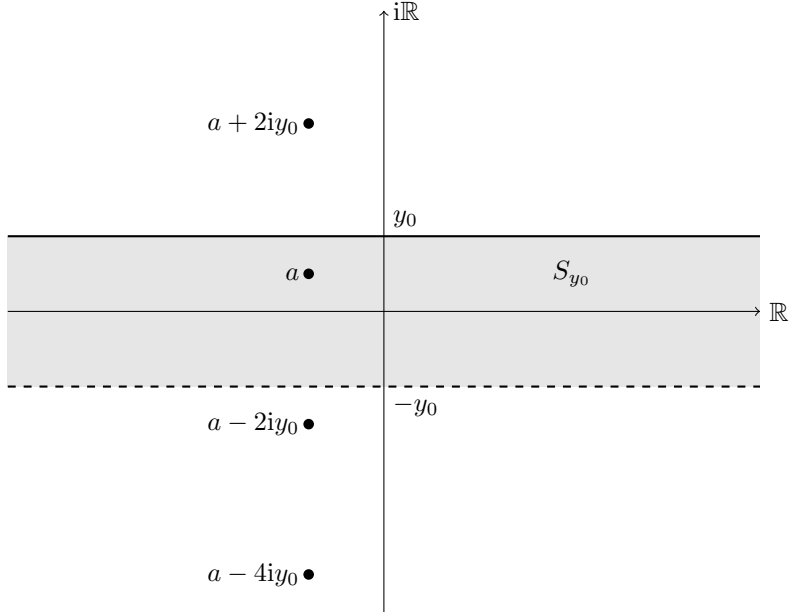
\begin{figure}
    \centering
    \begin{tikzpicture}

        \fill[gray!20] (-5,-1) rectangle (5,1);
        \node at (2.5,0.5) {$S_{y_0}$};

        \draw[->] (-5,0)--(5,0) node[right] {$\bbR$};
        \draw[->] (0,-4)--(0,4) node[right] {$\rmi \bbR$};

        \draw[thick, dashed] (-5,-1)--(5,-1);
        \draw[thick] (-5,1)--(5,1);
        \node[above right] at (0,1) {$y_0$};
        \node[below right] at (0,-1) {$-y_0$};

        \fill[black] (-1,.5) circle (2pt) node[left] {$a$};
        \fill[black] (-1,2.5) circle (2pt) node[left] {$a+2\rmi y_0$};
        \fill[black] (-1,-1.5) circle (2pt) node[left] {$a-2\rmi y_0$};
        \fill[black] (-1,-3.5) circle (2pt) node[left] {$a-4\rmi y_0$};

    \end{tikzpicture}
    \caption{Depiction of $S_{y_0}$. Note that the dashed line indicates that $\bbR - \rmi y_0$ is not contained in $S_{y_0}$.}
    \label{fig:parallel_splitting_zeroes}
\end{figure}

As for lines intersecting at a rational angle, we classify all finite order entire functions $f,g \in \calO(\bbC)$ whose magnitudes agree on two parallel lines: $\lvert f \rvert = \lvert g \rvert$ on $\bbR \cup (\bbR + \rmi y_0)$. Let us continue to consider the splitting of the zeroes of $f$ into those that $f$ shares with $g$ (described by the multiset $M_\mathrm{c}$) and those that it does not (described by the multiset $M_\mathrm{d}$). We can further subdivide the zeroes of $f$ that are no zeroes of $g$ because $M_\mathrm{d}$ is invariant under translations. Indeed, let $\sigma_{y_0}$ denote reflection about $\bbR + \rmi y_0$ and let $\sigma_0$ denote complex conjugation. Then, equations~\eqref{eq:mirror_symmetry} and \eqref{eq:mirror_symmetry_general} together imply that $M_\mathrm{d} = M_\mathrm{d} \circ \tau_{2 y_0}^k$ for all $k \in \bbZ$, where $\tau_{2 y_0} = \sigma_0 \circ \sigma_{y_0}$ denotes translation by $2 \rmi y_0$. It follows that the values of $M_\mathrm{d}$ are completely determined by its values on the (infinite) strip (cf.~Figure~\ref{fig:parallel_splitting_zeroes})
\begin{equation*}
    S_{y_0} := \{ z \in \bbC \,|\, \Im z \in (-y_0,y_0] \}.
\end{equation*}
In particular, we have 
\[
    \supp(M_\mathrm{d}) = \tau_{2 y_0}^{\bbZ} (\supp(M_\mathrm{d}) \cap S_{y_0}).
\]

Interestingly, the above implies that $\supp(M_\mathrm{d})$ contains an arithmetic sequence if it is non-empty. It follows, by an application of Jensen's formula, that $f$ and $g$ must be entire functions of order at least one in this case\footnote{Indeed, suppose that $f$ is an entire function of order $\rho < 1$ and pick $\epsilon > 0$ such that $\rho + \epsilon < 1$. Then, a classical application of Jensen's theorem (cf.~the argument in \cite[p.~249]{titchmarsh1939theory} for instance) implies that the number $n(r)$ of zeroes $a$ of $f$ with $\lvert a \rvert \leq r$ satisfies $n(r) \lesssim r^{\rho + \epsilon}$ where the implicit constant only depends on $\epsilon$. Hence, the set of zeroes of $f$ does \emph{not} contain an arithmetic sequence.}. We now establish the following theorem.

\begin{theorem}\label{thm:parallel}
    Let $y_0 > 0$ and let $f,g \in \calO(\bbC) \setminus \{0\}$ be of finite order. Then, $\abs{f} = \abs{g}$ on $\bbR \cup (\bbR + \rmi y_0)$ if and only if $f$ and $g$ factor as
    \begin{gather*}
        f(z) = \rme^{Q_f(z)} P_\mathrm{c}(z) \prod_{\substack{a \in \operatorname{supp}(M_\text{d}) \\ \Im a \in (-y_0,y_0]}} \prod_{k \in \bbZ} E\left( \frac{z}{\tau_{2 y_0}^k a} ; p \right)^{M_\text{d}(a)}, \\
        g(z) = \rme^{Q_g(z)} P_\mathrm{c}(z) \prod_{\substack{a \in \operatorname{supp}(M_\text{d}) \\ \Im a \in (-y_0,y_0]}} \prod_{k \in \bbZ} E\left( \frac{z}{\tau_{2 y_0}^k \overline a} ; p \right)^{M_\text{d}(a)}.
    \end{gather*}
    Here, $p \in \bbN_0$ is the genus of the canonical product formed with the zeroes of $f$, $P_\mathrm{c}$ is the canonical product formed with the common zeroes of $f$ and $g$, and $Q_f, Q_g \in \bbC[z]$ are polynomials of the same order $d \in \bbN_0$, whose coefficients have the same real parts and whose imaginary parts $(\lambda_\ell)_{\ell = 0}^d, (\mu_\ell)_{\ell = 0}^d \in \bbR$ satisfy
    \begin{multline*}
        \sum_{\ell = 1}^d (\lambda_\ell - \mu_\ell) \Im[ (x + \rmi y_0)^\ell ] = \sum_{\substack{a \in \operatorname{supp}(M_\text{d}) \\ \Im a \in (-y_0,y_0]}} M_\mathrm{d}(a) \sum_{k \in \bbZ} \Bigg( \log \abs{ 1 - \frac{2 \rmi y_0}{\tau_{2y_0}^k a} } \\
        + \sum_{\ell = 1}^p \ell^{-1} \Re \Bigg[ \left( \frac{x + \rmi y_0}{\tau_{2y_0}^k a} \right)^\ell - \left( \frac{x + \rmi y_0}{\tau_{2y_0}^{-k} \overline a + 2 \rmi y_0} \right)^\ell \Bigg] \Bigg),
    \end{multline*}
    for all $x \in \bbR$. 
\end{theorem}

\begin{proof}
    We show that $f$ and $g$ satisfying $\abs{f} = \abs{g}$ on $\bbR \cup (\bbR + \rmi y_0)$ can be factored as described above. As in the proof of Theorem~\ref{thm:rational_intersection}, we have that Theorem~\ref{thm:zeroflipping} along with the considerations above show that 
    \begin{gather*}
        f(z) = \rme^{Q_f(z)} P_\mathrm{c}(z) \prod_{\substack{a \in \operatorname{supp}(M_\text{d}) \\ \Im a \in (-y_0,y_0]}} \prod_{k \in \bbZ} E\left( \frac{z}{\tau_{2 y_0}^k a} ; p \right)^{M_\text{d}(a)}, \\
        g(z) = \rme^{Q_g(z)} P_\mathrm{c}(z) \prod_{\substack{a \in \operatorname{supp}(M_\text{d}) \\ \Im a \in (-y_0,y_0]}} \prod_{k \in \bbZ} E\left( \frac{z}{\tau_{2 y_0}^k \overline a} ; p \right)^{M_\text{d}(a)},
    \end{gather*}
    where $p \in \bbN_0$ is the genus of the canonical product formed with the zeroes of $f$, $P_\mathrm{c}$ is the canonical product formed with the common zeroes of $f$ and $g$, and $Q_f, Q_g \in \bbC[z]$ are polynomials of the same order whose coefficients have the same real parts.

    By rearranging the product over $k$ in the factorisation of $g$, we can see that 
    \begin{align*}
        \MoveEqLeft[3] \frac{\lvert f(x + \rmi y_0) \rvert}{\lvert g(x + \rmi y_0) \rvert} \\
        ={}& \rme^{\Re[ Q_f(x + \rmi y_0) - Q_g(x + \rmi y_0) ]} \prod_{\substack{a \in \operatorname{supp}(M_\text{d}) \\ \Im a \in (-y_0,y_0]}} \prod_{k \in \bbZ} \Bigg\lvert \frac{E\left( \frac{x + \rmi y_0}{\tau_{2 y_0}^k a} ; p \right)}{E\left( \frac{x + \rmi y_0}{\tau_{2 y_0}^{1-k} \overline a} ; p \right)} \Bigg\rvert^{M_\text{d}(a)} \\
        ={}& \rme^{\Re[ Q_f(x + \rmi y_0) - Q_g(x + \rmi y_0) ]} \prod_{\substack{a \in \operatorname{supp}(M_\text{d}) \\ \Im a \in (-y_0,y_0]}} \prod_{k \in \bbZ} \abs{ 1 - \frac{2 \rmi y_0}{\tau_{2y_0}^k a} }^{M_\text{d}(a)} \\
        & \cdot \exp\left( M_\mathrm{d}(a) \sum_{\ell = 1}^p \ell^{-1} \Re \left[ \left( \frac{x + \rmi y_0}{\tau_{2y_0}^k a} \right)^\ell - \left( \frac{x + \rmi y_0}{\tau_{2y_0}^{-k} \overline a + 2\rmi y_0} \right)^\ell \right] \right).
    \end{align*}
    It follows that 
    \begin{multline*}
        \sum_{\ell = 1}^d (\lambda_\ell - \mu_\ell) \Im[ (x + \rmi y_0)^\ell ] = \sum_{\substack{a \in \operatorname{supp}(M_\text{d}) \\ \Im a \in (-y_0,y_0]}} M_\mathrm{d}(a) \sum_{k \in \bbZ} \Bigg( \log \abs{ 1 - \frac{2 \rmi y_0}{\tau_{2y_0}^k a} } \\
        + \sum_{\ell = 1}^p \ell^{-1} \Re \Bigg[ \left( \frac{x + \rmi y_0}{\tau_{2y_0}^k a} \right)^\ell - \left( \frac{x + \rmi y_0}{\tau_{2y_0}^{-k} \overline a + 2 \rmi y_0} \right)^\ell \Bigg] \Bigg).
    \end{multline*}
    The reverse implication follows from the same arguments.
\end{proof}

\begin{remark}
    Let us make three remarks on this theorem:
    \begin{enumerate}
        \item The conditions on the coefficients of the polynomials $Q_f$ and $Q_g$ do not seem immediately useful. However, we will see in the following section that they simplify in some interesting settings.
        \item Using rigid motions, we can extend Theorem~\ref{thm:parallel} to arbitrary parallel lines $\ell_1$ and $\ell_2$. We can also extend it to sets of uniqueness $u_1 \subseteq \ell_1$ and $u_2 \subseteq \ell_2$ for finite order entire functions.
        \item According to Proposition~\ref{prop:expansion}, we can express the products over $k$ in the factorisations of $f$ and $g$ as 
        \begin{gather*}
            \frac{\sinh \big( \frac{\pi(a-z)}{2 y_0} \big)}{\sinh \big( \frac{\pi a}{2 y_0} \big)} \exp\bigg( \frac{\pi z}{2 y_0} \coth \Big( \frac{\pi a}{2 y_0} \Big) + \sum_{\ell = 2}^\infty \frac{z^\ell}{\ell} \sum_{k \in \bbZ} (\tau_{2y_0}^k a)^{-\ell} \bigg), \\
            \frac{\sinh \big( \frac{\pi(\overline a-z)}{2 y_0} \big)}{\sinh \big( \frac{\pi \overline a}{2 y_0} \big)} \exp\bigg( \frac{\pi z}{2 y_0} \coth \Big( \frac{\pi \overline a}{2 y_0} \Big) + \sum_{\ell = 2}^\infty \frac{z^\ell}{\ell} \sum_{k \in \bbZ} (\tau_{2y_0}^k \overline a)^{-\ell} \bigg).
        \end{gather*}
        We will use this insight in the following section.
    \end{enumerate}
\end{remark}

\section{Further results}

\subsection{Uniqueness for three parallel lines}

Consider three distinct parallel lines $\ell_1, \ell_2, \ell_3 \subset \mathbb{C}$ on which the magnitudes of two non-zero entire functions $f$ and $g$ agree. We denote the distance between $\ell_1$ and $\ell_2$ by $a > 0$ and the distance between $\ell_2$ and $\ell_3$ by $b > 0$. Assume, without loss of generality, that the three lines are parallel to the real numbers $\bbR$. Then, the zeroes of $f$ that are \emph{no} zeroes of $g$, described by the multiset $M_\mathrm{d}$, are invariant under the translations $\tau_{2a}$ and $\tau_{2b}$ by $2 \rmi a$ and $2 \rmi b$, respectively.

If $\tfrac{b}{a}$ is irrational, then $f$ and $g$ have the same zeroes: indeed, if we assume that there exists a $z \in \mathbb{C}$ such that $M_\mathrm{d}(z) > 0$, then all elements of $z + 2\mathrm{i}a\mathbb{Z} + 2\mathrm{i}b\mathbb{Z}$ are zeroes of $f$. Since $\tfrac{b}{a}$ is irrational, $z + 2\mathrm{i}a\mathbb{Z} + 2\mathrm{i}b\mathbb{Z}$ is dense in $z + \mathrm{i}\mathbb{R}$ by Kronecker's theorem. Therefore, $f = 0$: a contradiction.

We can now conclude that $f$ and $g$ agree up to global phase.

\begin{theorem}\label{thm:three_lines}
    Let $\ell_1, \ell_2, \ell_3 \subset \mathbb{C}$ be three distinct parallel lines. Let $a > 0$ denote the distance between $\ell_1$ and $\ell_2$ and let $b > 0$ denote the distance between $\ell_2$ and $\ell_3$. Assume that $\tfrac{b}{a}$ is irrational and let $f,g \in \calO(\bbC)$ be of finite order. Then, $\lvert f \rvert = \lvert g \rvert$ on $\ell_1 \cup \ell_2 \cup \ell_3$ if and only if $f \sim g$.
\end{theorem}

\begin{proof}
    Suppose that $\lvert f \rvert = \lvert g \rvert$ on $\ell_1 \cup \ell_2 \cup \ell_3$. As explained before the statement of the theorem, we have that $M_\mathrm{d} = 0$. Let us assume, without loss of generality, that one of the lines coincides with the real numbers $\mathbb{R}$. Then, Theorem~\ref{thm:zeroflipping} shows that $f = \mathrm{e}^{Q_f} P_\mathrm{c}$ and $g = \mathrm{e}^{Q_g} P_\mathrm{c}$, with polynomials $Q_f, Q_g \in \mathbb{C}[z]$ whose coefficients have the same real parts.

    Denote the imaginary parts of those coefficients by $(\lambda_k)_{k = 0}^d, (\mu_k)_{k = 0}^d \in \bbR$ and let $y_0 \in \mathbb{R} \setminus \{0\}$ be such that $\mathbb{R} + \mathrm{i}y_0$ is one of the parallel lines. Then, $\lvert f \rvert = \lvert g \rvert$ on $\mathbb{R} + \mathrm{i}y_0$ implies $\Re Q_f = \Re Q_g$ on $\mathbb{R} + \mathrm{i}y_0$. Expressed in terms of the imaginary parts of the coefficients, this is 
    \begin{equation*}
        \sum_{k = 0}^d ( \lambda_k - \mu_k ) \Im[ (x + \mathrm{i}y_0)^k ] = 0, \quad x \in \mathbb{R}.
    \end{equation*}
    We expand the power $ (x + \mathrm{i}y_0)^k $ and obtain
    \begin{equation*}
        \sum_{\ell = 0}^d x^\ell \sum_{k = \ell}^d \binom{k}{\ell} \sigma_{k-\ell} y_0^{k-\ell} ( \lambda_k - \mu_k ) = 0,
    \end{equation*}
    where 
    \begin{equation*}
        \sigma_k := \Im[\mathrm{i}^k] = \begin{cases}
            (-1)^\ell & \text{if } k = 2\ell+1, \\
            0 & \text{else}.
        \end{cases}
    \end{equation*}
    Comparing coefficients yields
    \begin{equation}\label{eq:comparing_coefficients}
        \sum_{k = \ell}^d \binom{k}{\ell} \sigma_{k-\ell} y_0^{k-\ell} ( \lambda_k - \mu_k ) = 0, \qquad \ell = 0,\dots,d.
    \end{equation}

    We finally show $ \lambda_k = \mu_k $ for all $ k = 1,\dots,d $ inductively, which proves the theorem. As base case, we show that $ \lambda_d = \mu_d $, which follows from equation~\eqref{eq:comparing_coefficients} with $ \ell = d-1 $. Now, suppose that $ \lambda_k = \mu_k $ for all $ k = j+1,\dots,d $, where $ j \in \{1,\dots,d-1\} $. Then, we consider $ \ell = j-1 $ in equation~\eqref{eq:comparing_coefficients} and obtain 
    \begin{align*}
        \MoveEqLeft[3] \sum_{k = j-1}^d \binom{k}{j-1} \sigma_{k-j+1} y_0^{k-j+1} ( \lambda_k - \mu_k ) \\
        ={}& j y_0 ( \lambda_{j} - \mu_{j} ) + \sum_{k = j+2}^d \binom{k}{j-1} \sigma_{k-j+1} y_0^{k-j+1} ( \lambda_k - \mu_k ) \\
        ={}& j y_0 ( \lambda_{j} - \mu_{j} ) = 0,
    \end{align*}
    such that $ \lambda_{j} = \mu_{j} $.
\end{proof}

\begin{remark}
    We note that the theorem extends to sets of uniqueness $ u_1 \subseteq \ell_1 $, $ u_2 \subseteq \ell_2 $, and $ u_3 \subseteq \ell_3 $ for the entire functions under consideration. We also note that the theorem remains true for entire functions that are \emph{not} of finite order. The argument is analogous to that presented in Remark~\ref{rem:reflection_principle}.
\end{remark}

\subsection{Characterisation of finite order entire functions with common magnitudes on equidistant parallel lines}

Our characterisation of finite order entire functions with common magnitudes on two parallel lines (see Theorem~\ref{thm:parallel}) naturally extends to a similar result on an arbitrary number of equidistant parallel lines. In the following, we present a simplified discussion and assume that $f$ and $g$, both in $\calO(\bbC)$, are of \emph{order less than two}. This is motivated by Gabor transform phase retrieval, which is connected to phase retrieval of entire functions with order less than two and type less than $\pi/2$ (see the explanation in Subsection~\ref{ssec:remark}).

We present a theorem characterising second order entire functions whose magnitudes agree on infinitely many equidistant parallel lines.

\begin{theorem}\label{thm:infinite}
    Let $y_0 > 0$ and let $f,g \in \calO(\bbC) \setminus \{0\}$ be of order less than two. Then, $\abs{f} = \abs{g}$ on $\bbR + \rmi y_0 \bbZ$ if and only if $f$ and $g$ factor as 
    \begin{gather*}
        f(z) = \rme^{Q_f(z)} P_\mathrm{c}(z) \prod_{\substack{a \in \operatorname{supp}(M_\text{d})\\\Im a \in (-y_0,y_0]}} \left[ \frac{\sinh \left( \frac{\pi (a-z)}{2 y_0} \right)}{\sinh\left( \frac{\pi a}{2 y_0} \right)} \rme^{R_a(z)} \right]^{M_\mathrm{d}(a)}, \\
        g(z) = \rme^{Q_g(z)} P_\mathrm{c}(z) \prod_{\substack{a \in \operatorname{supp}(M_\text{d})\\\Im a \in (-y_0,y_0]}} \left[ \frac{\sinh \left( \frac{\pi (\overline a-z)}{2 y_0} \right)}{\sinh\left( \frac{\pi \overline a}{2 y_0} \right)} \rme^{R_{\bar a}(z)} \right]^{M_\mathrm{d}(a)}.
    \end{gather*}
    Here, $Q_f, Q_g \in \bbC_2[z]$ are polynomials, whose coefficients have the same real parts and whose imaginary parts $(\lambda_j)_{j = 0}^2, (\mu_j)_{j = 0}^2 \in \bbR$ satisfy
    \begin{gather*}
        \mu_1 - \lambda_1 = \frac{\pi}{y_0} \sum_{\substack{a \in \supp(M_\mathrm{d}) \\ \Im a \in (-y_0,y_0]}} M_\mathrm{d}(a) \Im \coth \left( \frac{\pi a}{2 y_0} \right), \\
        \mu_2 - \lambda_2 = \frac{\pi^2}{4 y_0^2} \sum_{\substack{a \in \supp(M_\mathrm{d}) \\ \Im a \in (-y_0,y_0]}} M_\mathrm{d}(a) \Im \csch^2 \left( \frac{\pi a}{2 y_0} \right),
    \end{gather*}
    $P_\mathrm{c}$ is the canonical product formed with the common zeroes of $f$ and $g$, and
    \begin{equation*}
        R_a(z) := R_{a,y_0}(z) := \frac{\pi z}{2 y_0} \coth\left( \frac{\pi a}{2 y_0} \right) + \frac{\pi^2 z^2}{8 y_0^2} \csch^2 \left( \frac{\pi a}{2 y_0} \right).
    \end{equation*}
\end{theorem}

\begin{proof}
    By appropriately rescaling $f$ and $g$, we can assume that $y_0 = \tfrac{\pi}{2}$ without loss of generality. Let us thereby assume that $\lvert f \rvert = \lvert g \rvert$ on $\bbR + \tfrac{\pi \rmi}{2} \bbZ$. As in the proof of Theorem~\ref{thm:parallel}, we can use Theorem~\ref{thm:zeroflipping} to see that $f$ and $g$ factor as 
    \begin{gather*}
        f(z) = \rme^{Q_f(z)} P_\mathrm{c}(z) \prod_{\substack{a \in \operatorname{supp}(M_\text{d}) \\ \Im a \in (-\pi/2,\pi/2]}} \prod_{k \in \bbZ} E\left( \frac{z}{\tau_{\pi}^k a} ; 2 \right)^{M_\text{d}(a)}, \\
        g(z) = \rme^{Q_g(z)} P_\mathrm{c}(z) \prod_{\substack{a \in \operatorname{supp}(M_\text{d}) \\ \Im a \in (-\pi/2,\pi/2]}} \prod_{k \in \bbZ} E\left( \frac{z}{\tau_{\pi}^k \overline a} ; 2 \right)^{M_\text{d}(a)}.
    \end{gather*}
    Here, we may note that the genus of the canonical product formed with the zeroes of $f$ is at most two (see e.g.~\cite[Subsection~8.23 on p.~250]{titchmarsh1939theory}). Additionally, the polynomials $Q_f, Q_g \in \bbC_2[z]$ have degree at most two according to the Hadamard factorisation theorem. Remember also that the coefficients of $Q_f$ and $Q_g$ have the same real parts and that $P_\mathrm{c}$ denotes the canonical product formed with the common zeroes of $f$ and $g$.

    By Proposition~\ref{prop:expansion} and the partial fraction expansion of the second power of the cosecant (see e.g.~\cite[Example~(iv) on p.~113]{titchmarsh1939theory}), we can write 
    \begin{align*}
        \prod_{k \in \bbZ} E\left( \frac{z}{\tau_{\pi}^k a} ; 2 \right) &= \prod_{k \in \bbZ} E\left( \frac{z}{\tau_{\pi}^k a} ; 1 \right) \cdot \exp\left( \frac{z^2}{2} \sum_{k \in \bbZ} (\tau_{\pi}^k a)^{-2} \right) \\
        &= \frac{\sinh ( a-z )}{\sinh a} \exp( z \coth a ) \cdot \exp\left( - \frac{z^2}{2} \csc^2 (\rmi a) \right)
    \end{align*}
    for $a \in \bbC \setminus \pi \rmi \bbZ$. Therefore, 
    \begin{gather*}
        \prod_{\substack{a \in \operatorname{supp}(M_\text{d}) \\ \Im a \in (-\pi/2,\pi/2]}} \prod_{k \in \bbZ} E\left( \frac{z}{\tau_{\pi}^k a} ; 2 \right)^{M_\text{d}(a)} = \prod_{\substack{a \in \operatorname{supp}(M_\text{d}) \\ \Im a \in (-\pi/2,\pi/2]}} \left[ \frac{\sinh ( a-z )}{\sinh a} \rme^{R_a(z)} \right]^{M_\text{d}(a)}, \\
        \prod_{\substack{a \in \operatorname{supp}(M_\text{d}) \\ \Im a \in (-\pi/2,\pi/2]}} \prod_{k \in \bbZ} E\left( \frac{z}{\tau_{\pi}^k \overline a} ; 2 \right)^{M_\text{d}(a)} = \prod_{\substack{a \in \operatorname{supp}(M_\text{d}) \\ \Im a \in (-\pi/2,\pi/2]}} \left[ \frac{\sinh ( \overline a-z )}{\sinh \overline a} \rme^{R_{\bar a}(z)} \right]^{M_\text{d}(a)}.
    \end{gather*}

    It remains for us to consider the polynomials $Q_f$ and $Q_g$: let $x \in \bbR$ and compute
    \begin{multline*}
        \frac{\lvert f(x-\frac{\pi \rmi}{2}) \rvert}{\lvert g(x-\frac{\pi \rmi}{2}) \rvert} = \rme^{\Re[ Q_f(x - \frac{\pi \rmi}{2}) - Q_g(x - \frac{\pi \rmi}{2}) ]} \\
        \cdot \prod_{\substack{a \in \operatorname{supp}(M_\text{d}) \\ \Im a \in (-\pi/2,\pi/2]}} \left( \frac{\lvert \sinh \overline a \rvert}{\lvert \sinh a \rvert} \frac{\lvert \sinh (a-z + \frac{\pi \rmi}{2}) \rvert}{\lvert \sinh (\overline a - z + \frac{\pi \rmi}{2}) \rvert} \right)^{M_\mathrm{d}(a)} \\
        \cdot \rme^{M_\mathrm{d}(a) \Re[ R_a(x - \frac{\pi \rmi}{2}) - R_{\bar a}(x - \frac{\pi \rmi}{2}) ]}.
    \end{multline*}
    Since
    \begin{equation*}
        \sinh \overline a = \overline{ \sinh a }, \qquad \sinh(\overline a - x + \tfrac{\pi \rmi}{2}) = - \overline{\sinh(a-x+\tfrac{\pi \rmi}{2})},
    \end{equation*}
    and $\lvert f \rvert = \lvert g \rvert$ on $\bbR - \tfrac{\pi \rmi}{2}$, the above simplifies to 
    \begin{multline*}
        \Re\left[ Q_f\left(x - \frac{\pi \rmi}{2}\right) - Q_g\left(x - \frac{\pi \rmi}{2}\right) \right] \\
        = - \sum_{\substack{a \in \operatorname{supp}(M_\text{d}) \\ \Im a \in (-\pi/2,\pi/2]}} M_\mathrm{d}(a) \Re\left[ R_a\left(x - \frac{\pi \rmi}{2}\right) - R_{\bar a}\left(x - \frac{\pi \rmi}{2}\right) \right].
    \end{multline*}

    We may now use the following simple claim

    \begin{proof}[Claim.]\renewcommand{\qedsymbol}{}
        Let $p,q \in \bbC_2[z]$ be two polynomials whose coefficients have the same real parts. Denote the imaginary parts of their coefficients by $(a_j)_{j = 0}^2,(b_j)_{j = 0}^2 \in \bbR$. Then, 
        \begin{equation*}
            \Re[ p(x-\rmi y) - q(x-\rmi y) ] = 2(a_2 - b_2) xy - (a_1 - b_2) y, \qquad x,y \in \bbR.
        \end{equation*}
    \end{proof}

    It follows that
    \begin{multline*}
        \pi (\lambda_2 - \mu_2) x - \frac{\pi}{2} (\lambda_1 - \mu_2)  \\
        = - \pi \sum_{\substack{a \in \operatorname{supp}(M_\text{d}) \\ \Im a \in (-\pi/2,\pi/2]}} M_\mathrm{d}(a) ( x \Im \csch^2 a - \Im \coth a )
    \end{multline*}
    and the factorisation is complete. The reverse implication is shown in a similar way.
\end{proof}

\begin{example}
    The functions $f,g \in \calO(\bbC)$ defined in equation~\eqref{eq:counterexamples} are recovered with 
    \begin{equation*}
        Q_f(z) = z^2 + \rmi z,~Q_g(z) = z^2 - \rmi z,~P_\mathrm{c} = 1,~M_\mathrm{d} = \chi_{\pi \rmi \bbZ + \frac{\pi \rmi}{4}},~y_0 = \frac{\pi}{2},
    \end{equation*}
    where $\chi_{\pi \rmi \bbZ + \pi \rmi / 4}$ is the indicator function of the set $\pi \rmi \bbZ + \tfrac{\pi \rmi}{4}$. It follows directly from Theorem~\ref{thm:infinite} that $\lvert f \rvert = \lvert g \rvert$ on $\bbR + \tfrac{\pi \rmi}{2} \bbZ$. Additionally, it is readily seen that $f$ and $g$ have different sets of zeroes such that $f \not\sim g$.
\end{example}

\begin{remark}
    The periodicity of the zeroes of $f$ that are \emph{no} zeroes of $g$ is reminiscent of the work in \cite{groechenig2020phase}. Indeed, it turns out to be possible to show that $\abs{f} = \abs{g}$ on $\bbR + \rmi y_0 \bbZ$ if and only if $f$ and $g$ factor as 
    \begin{gather*}
        f(z) = r \rme^{\rmi \alpha} \rme^{Q(z)} P_\mathrm{c}(z) \prod_{a \in W_-} \left[ \frac{\rme^{\frac{\pi}{y_0}(a-z)}-1}{\rme^{\frac{\pi}{y_0}a}-1} \right]^{M_\mathrm{d}(a)} \prod_{a \in W_+} \left[ \frac{1 - \rme^{-\frac{\pi}{y_0}(a-z)}}{1 - \rme^{-\frac{\pi}{y_0}a}} \right]^{M_\mathrm{d}(a)}, \\
        g(z) = r \rme^{\rmi \beta} \rme^{Q(z)} P_\mathrm{c}(z) \prod_{a \in W_-} \left[ \frac{\rme^{\frac{\pi}{y_0}(\overline a-z)}-1}{\rme^{\frac{\pi}{y_0}\overline a}-1} \right]^{M_\mathrm{d}(a)} \prod_{a \in W_+} \left[ \frac{1 - \rme^{-\frac{\pi}{y_0}(\overline a-z)}}{1 - \rme^{-\frac{\pi}{y_0}\overline a}} \right]^{M_\mathrm{d}(a)}.
    \end{gather*}
    Here, $r > 0$, $\alpha,\beta \in \bbR$, $Q \in \bbC_2[z]$ is a polynomial with $Q(0) = 0$, $P_\mathrm{c}$ is the canonical product formed with the common zeroes of $f$ and $g$, and 
    \begin{gather*}
        W_+ := \set{ a \in [0,\infty) + \rmi (-y_0,y_0] }{ M_\mathrm{d}(a) > 0 }, \\
        W_- := \set{ a \in (-\infty,0) + \rmi (-y_0,y_0] }{ M_\mathrm{d}(a) > 0 }.
    \end{gather*}
    A proof of this is found in Appendix~\ref{app:infinite}.
\end{remark}

\subsection{Construction of universal counterexamples}

Given a set $\Omega \subset \bbC$ of infinitely many equidistant parallel lines, we have constructed entire functions $f,g \in \calO(\bbC)$ of order one which depend on $\Omega$, do \emph{not} agree up to global phase, and whose magnitudes agree on $\Omega$. We call such functions counterexamples to uniqueness for phase retrieval of entire functions; or, simply \emph{counterexamples}.

This raises the question whether $f$ can be constructed independently on $\Omega$ as well; i.e., whether we can form an entire function $f \in \calO(\bbC)$ such that, for all sets $\Omega$ of infinitely many equidistant parallel lines in some family $\calF$, we can find $g_\Omega \in \calO(\bbC)$ such that $f \not\sim g_\Omega$ while $\lvert f \rvert = \lvert g_\Omega \rvert$ on $\Omega$. We call the function $f$ \emph{universal counterexample} for $\calF$ if it exists.

For the family $\calF = \{ \bbR + \rmi a \bbZ \,|\, a > 0 \}$, \emph{no} universal counterexamples of finite order exist (if we require $g$ to be of finite order as well).

\begin{lemma}
    For all $f \in \calO(\bbC)$ of finite order, there exists a positive number $a > 0$, such that, for all $g \in \calO(\bbC)$ of finite order, $\lvert f \rvert = \lvert g \rvert$ on $\bbR + \rmi a \bbZ$ implies $f \sim g$. 
\end{lemma}

\begin{proof}
    Suppose by contradiction that the lemma is untrue at $f \in \calO(\bbC)$. Fix $a > 0$ and consider $g \in \calO(\bbC)$ of finite order such that $f \not\sim g$ while $\lvert f \rvert = \lvert g \rvert$ on $\bbR + \rmi a \bbZ$. Then, the multiset $M_\mathrm{d}$, describing the zeroes of $f$ that are \emph{no} zeroes of $g$, is invariant under the translation $\tau_{2a}$. At the same time, $M_\mathrm{d}$ is \emph{not} identically zero. Indeed, if it were, then Hadamard factorisation of $f$ and $g$ would yield $f = \rme^{Q_f} P$ and $g = \rme^{Q_g} P$, where $Q_f,Q_g \in \bbC[z]$ and $P$ is the canonical product formed with the zeroes of $f$. It then follows that $Q_f = Q_g + \rmi c$ for some $c \in \bbR$ as in the proof of Theorem~\ref{thm:three_lines}, which implies that $f \sim g$: a contradiction.

    Hence, we can find $z_a \in \bbC$ such that $M_\mathrm{d}(\tau_{2a}^k z_a) > 0$ for all $k \in \bbZ$. It follows that $\tau_{2 a}^\bbZ z_a$ is a subset of the zeroes of $f$. Since $a > 0$ is arbitrary, it follows that $f$ has uncountably many zeroes which implies that $f = 0$ and thus $g = 0$. We have reached a contradiction again and the lemma is proven.
\end{proof}

\begin{remark}
    The proof above does not use that $\calF = \{ \bbR + \rmi a \bbZ \,|\, a > 0 \}$ contains sets of parallel lines that are arbitrarily close. Instead only the uncountability of $\bbR_+$ is utilised. It follows that it is also true that, for all uncountable $A \subset \bbR_+$ and all $f \in \calO(\bbC)$ of finite order, there exists an $a \in A$, such that, for all $g \in \calO(\bbC)$ of finite order, $\lvert f \rvert = \lvert g \rvert$ on $\bbR + \rmi a \bbZ$ implies $f \sim g$.
\end{remark}

This raises the question whether the same remains true if we replace $A$ by a countable set containing arbitrarily small numbers such as $\bbN^{-1}$, for instance. We answer this question by considering a direct construction. 

\begin{theorem}
    Let $f : \bbN_0 \to (0,\infty)$ be such that $\lim_{n\to\infty} f(n) = \infty$ and define $a_n := \exp(\pi n f(n))$ for $n\in\bbN_0$. The functions 
    \begin{equation*}
        g_n(z) := \frac{a_n - \rmi \rme^{\pi n z}}{a_n - \rmi} \prod_{\substack{k \in \bbN\\k \neq n}} \frac{a_k + \rmi \rme^{\pi k z}}{a_k + \rmi}, \qquad n \in \bbN_0,
    \end{equation*}
    are well-defined entire functions and satisfy $\lvert g_0 \rvert = \lvert g_n \rvert$ on $\bbR + \tfrac{\rmi}{n} \bbZ$ for $n \in \bbN$. If $f$ is injective, then $g_0 \not\sim g_n$ for $n \in \bbN$, and, if $f = \exp$, then $g_n$ is of order one for $n \in \bbN_0$.
\end{theorem}

\begin{proof}
    Let $R > 0$ and let $K \in \bbN$ be such that $k \geq K$ implies $f(k) \geq R + \pi^{-1}$. Then, 
    \begin{equation*}
        \sum_{k \geq K} \abs{ \frac{a_k + \rmi \rme^{\pi k z}}{a_k + \rmi} - 1 } = \sum_{k \geq K} \frac{\lvert \rme^{\pi k z} - 1\rvert}{\lvert a_k + \rmi \rvert} \leq \sum_{k \geq K} \rme^{\pi k (\lvert z \rvert - f(k))} \leq \sum_{k \geq K} \rme^{-k} < \infty,
    \end{equation*}
    for all $\lvert z \rvert \leq R$, such that the infinite product converges locally uniformly to an entire function. 

    Next, let $n \in \bbN$, $m \in \bbZ$, and $x \in \bbR$. We can directly compute that $g_0(f(n) + \tfrac{\rmi}{2n}) = 0$ and 
    \begin{equation*}
        \frac{\lvert g_0(x+\frac{m \rmi}{n}) \rvert}{\lvert g_n(x+\frac{m \rmi}{n}) \rvert} = \frac{\abs{\frac{a_n + \rmi \rme^{\pi n (x + m/n \rmi)}}{a_n + \rmi}}}{\abs{\frac{a_n - \rmi \rme^{\pi n (x + m/n \rmi)}}{a_n - \rmi}}} = \frac{\abs{a_n + \rmi (-1)^m \rme^{\pi n x}}}{\abs{a_n - \rmi (-1)^m \rme^{\pi n x}}} = 1.
    \end{equation*}

    If we assume that $g_n(f(n) + \tfrac{\rmi}{2n}) = 0$, then there exists $k \in \bbN$, $k \neq n$, such that
    \begin{equation*}
        a_k + \rmi \rme^{\pi k (f(n) + \frac{\rmi}{2n})} = 0 \iff f(k) = f(n) \mbox{ and } \frac{k+n}{4n} \in \bbZ.
    \end{equation*}
    If $f$ is injective, the above is impossible and we can conclude that $g_0 \not\sim g_n$. 

    Finally suppose that $f = \exp$. We show that $g_0$ is of order one by splitting the product over $k \in \bbN$ into a finite and an infinite product:
    \begin{equation*}
        g_0(z) = \prod_{\substack{k \in \bbN\\k \leq \log (\lvert z \rvert + \pi^{-1})}} \frac{a_k + \rmi \rme^{\pi k z}}{a_k + \rmi} \cdot \prod_{\substack{k \in \bbN\\k > \log (\lvert z \rvert + \pi^{-1})}} \frac{a_k + \rmi \rme^{\pi k z}}{a_k + \rmi}.
    \end{equation*}
    For the finite product, we note that 
    \begin{equation*}
        \abs{\frac{a_k + \rmi \rme^{\pi k z}}{a_k + \rmi}} \leq a_k + \rme^{\pi k \lvert z \rvert} \leq (\lvert z \rvert + \pi^{-1} + 1) \exp( \pi \lvert z \rvert \log (\lvert z \rvert + \pi^{-1}) ).
    \end{equation*}
    For the infinite product, we rewrite 
    \begin{equation*}
        \frac{a_k + \rmi \rme^{\pi k z}}{a_k + \rmi} = 1 - \frac{\rmi(1 - \rme^{\pi k z})}{a_k + \rmi} =: 1 - \phi_k(z).
    \end{equation*}
    Note that $\lvert \phi_k(z) \rvert \leq \rme^{\pi k (\lvert z \rvert - \exp(k))} < \rme^{-k} < 1$ such that an application of the Taylor series for $\log(1-z)$ yields 
    \begin{equation*}
        \log \Bigg\lvert \prod_{\substack{k \in \bbN\\k > \log \lvert z \rvert}} \frac{a_k + \rmi \rme^{\pi k z}}{a_k + \rmi} \Bigg\rvert = - \sum_{\substack{k \in \bbN\\k > \log \lvert z \rvert}} \sum_{n \in \bbN} \frac{\Re[\phi_k(z)^n]}{n} \leq \sum_{k,n \in \bbN} \rme^{-nk} =:c  < \infty.
    \end{equation*}
    It follows that 
    \begin{equation*}
        \lvert g_0(z) \rvert \leq (\lvert z \rvert + \pi^{-1} + 1)^{\log (\lvert z \rvert + \pi^{-1})} \exp( \pi \lvert z \rvert \log^2 (\lvert z \rvert + \pi^{-1}) ) \cdot \rme^c 
    \end{equation*}
    such that $g_0$ is of first order. Therefore, the functions $(g_n)_{n \in \bbN}$ are of first order as well.
\end{proof}

\begin{remark}
    The function $g_0$ was initially constructed by choosing the zeroes $z_{n,k} := \rme^n + \tfrac{\rmi}{2n} + \tfrac{2\rmi k}{n}$, for $n \in \bbN$, $k \in \bbZ$, and setting up the canonical product 
    \begin{equation*}
        \prod_{n \in \bbN} \prod_{k \in \bbZ} E\left( \frac{z}{z_{n,k}};1 \right).
    \end{equation*}
    By Proposition~\ref{prop:expansion}, the above coincides with 
    \begin{equation*}
        \prod_{n \in \bbN} \frac{\sinh\left( \frac{\pi n}{2} (\rme^n + \frac{\rmi}{2n} - z) \right)}{\sinh\left( \frac{\pi n}{2} (\rme^n + \frac{\rmi}{2n}) \right)} \rme^{\frac{\pi n}{2} \coth\left( \frac{\pi n}{2} (\rme^n + \frac{\rmi}{2n}) \right) z}.
    \end{equation*}
    In exponential form, this is 
    \begin{equation*}
        \prod_{n \in \bbN} \frac{a_n \rme^{-\pi n z} + \rmi}{a_n + \rmi} \exp\left( \frac{\pi n a_n z}{a_n + \rmi} \right),
    \end{equation*}
    where we borrow the notation $a_n = \exp(\pi n \rme^n)$ from the theorem. Multiplying $(a_n \rme^{-\pi n z} + \rmi)/(a_n + \rmi)$ by $\rme^{\pi n z}$ yields fractions whose product converges to an entire function as shown in the proof above.
\end{remark}

\subsection{A remark on Gabor transform phase retrieval}\label{ssec:remark}

\subsubsection*{Introducing the Gabor and Bargmann transform}

Gabor transform phase retrieval and phase retrieval of entire functions are related by the connection of the Gabor and Bargmann transforms. Specifically, the \emph{Gabor transform} of a square-integrable function $f \in L^2(\bbR)$ is 
\[
    \calG f (x,\omega) := 2^{1/4} \int_\bbR f(t) \rme^{-\pi(t-x)^2} \rme^{-2\pi\rmi t \omega} \dd t, \qquad (x,\omega) \in \bbR^2,
\]
while the \emph{Bargmann transform} is 
\[
    \calB f (z) := 2^{1/4} \int_\bbR f(t) \rme^{2\pi t z - \pi t^2 - \frac{\pi}{2} z^2} \dd t, \qquad z \in \bbC.
\]

The Bargmann transform is a unitary map from $L^2(\bbR)$ onto the Fock space $\calF^2(\bbC)$ of entire functions which are square-integrable with respect to the Gaussian probability measure \cite[Theorem~3.4.3 on p.~56]{groechenig2001foundations}, given by 
\begin{equation*}
    \gamma(A) := \int_A \rme^{-\pi\lvert z \rvert^2} \dd z, \qquad A \subset \bbC.
\end{equation*}

Interestingly, the Fock space is a reproducing kernel Hilbert space with the property \cite[Theorem~3.4.2 on p.~54]{groechenig2001foundations}
\begin{equation*}
    \abs{F(z)} \leq \norm{F}_{L^2(\bbC,\gamma)} \rme^{\pi \abs{z}^2 / 2}, \qquad z \in \bbC,
\end{equation*}
from which we can immediately conclude that $F \in \calF^2(\bbC)$ is of order less than two and type less than $\tfrac{\pi}{2}$. Additionally, it is clear that any entire function of order strictly less than two (or of order two and type strictly less than $\tfrac{\pi}{2}$) is in the Fock space.

Finally, the Gabor and Bargmann transforms satisfy \cite[Proposition~3.4.1 on p.~54]{groechenig2001foundations}
\begin{equation*}
    \calG f (x,-\omega) = \rme^{\pi \rmi x \omega} \calB f (z) \rme^{-\pi \abs{z}^2/2}, \qquad z = x+\rmi \omega \in \bbC.
\end{equation*}
This relates Gabor transform phase retrieval --- which is the recovery of square-integrable signals $f$ from partial knowledge of the magnitudes of their Gabor transforms $\lvert \calG f \rvert$ on $\Omega \subseteq \bbR^2$ --- to the phase retrieval of entire functions of order at most two and type at most $\tfrac{\pi}{2}$. 

\subsubsection*{Three results on Gabor transform phase retrieval}

It follows that all the results that we have proven so far continue to hold for Gabor phase retrieval after suitable modifications. In the following, we illustrate this with three interesting examples. 

First, we see that Theorem~\ref{thm:three_lines} implies the following counterpart.

\begin{theorem}
    Let $\ell_1,\ell_2,\ell_3 \subset \bbR^2$ be three parallel lines. Let $a > 0$ denote the distance between $\ell_1$ and $\ell_2$ and let $b > 0$ denote the distance between $\ell_2$ and $\ell_3$. Assume that $\tfrac{b}{a}$ is irrational and let $f,g \in L^2(\bbR)$. Then, $\lvert \calG f \rvert = \lvert \calG g \rvert$ holds on $\ell_1 \cup \ell_2 \cup \ell_3$ if and only if $f \sim g$.
\end{theorem}

Secondly, note that the entire functions defined in equation~\eqref{eq:counterexamples} are of first order and thus in the Fock space $\calF^2(\bbC)$. Taking the inverse Bargmann transform yields 
\begin{gather*}
    f(t) = 2^{1/4} \rme^{-\frac{1}{2 \pi} - \pi t^2} \left( \cosh(2 t) + \rmi \sinh (2t) \right), \\
    g(z) = 2^{1/4} \rme^{-\frac{1}{2 \pi} - \pi t^2} \left( \cosh(2 t) - \rmi \sinh (2t) \right).
\end{gather*}
It follows from the relation of the Gabor and the Bargmann transform that $\lvert \calG f \rvert = \lvert \calG g \rvert$ on $\bbR \times \tfrac{\pi}{2} \bbZ$ while $f \not\sim g$. The functions above inspired the paper \cite{alaifari2021phase}.

Thirdly, there exist universal counterexamples for Gabor phase retrieval.

\begin{theorem}
    There exists a sequence $(f_n)_{n \in \bbN_0} \in L^2(\bbR)$ such that $f_0 \not\sim f_n$ and $\lvert \calG f_0 \rvert = \lvert \calG f_n \rvert$ on $\bbR \times \tfrac{\bbZ}{n}$ for all $n \in \bbN$. 
\end{theorem}

\subsubsection*{A bound on the number of zeroes in certain disks for functions in the Fock space}

In some applications of the results presented so far, one might want to construct a square-integrable function, whose Gabor transform has a specified set of zeroes $Z \subset \bbR^2$. This boils down to the construction of an entire function in the Fock space with the related set of zeroes $\{ z = x + \rmi y \in \bbC \,|\, (x,-y) \in Z \}$. In order for the latter to be achievable, we need to make sure that $Z$ is sufficiently spread out. 

Let us illustrate this with an interesting example: suppose we are given $a,b > 0$ and want to construct a non-zero function $f \in L^2(\bbR)$ such that $\calG f = 0$ on $a \bbZ \times b \bbZ$. A natural way of approaching this problem is to define 
\begin{equation}\label{eq:construction}
    F(z) = z \prod_{(m,n) \in \bbZ^2 \setminus \{(0,0)\}} E\left( \frac{z}{a m + \rmi b n} ; 2 \right), \qquad z \in \bbC,
\end{equation}
which is a well-defined entire function of order two according to \cite[Subection~8.25 on pp.~251--252]{titchmarsh1939theory} since
\begin{equation*}
    \sum_{m,n \geq 2} \frac{1}{(a^2 m^2 + b^2 n^2)^{1+\epsilon}} \leq \int_1^\infty \int_1^\infty \frac{1}{(a^2 x^2 + b^2 y^2)^{1+\epsilon}} \dd x \dd y < \infty,
\end{equation*}
for $\epsilon > 0$.

If we could show that the type of $F$ is strictly smaller than $\tfrac{\pi}{2}$, then $F \in \calF^2(\bbC)$ and we could write $f = \calB^{-1} F \in L^2(\bbR)$. Upper bounding the type of a product of primary factors is generally hard, however, and we will instead present a bound on the number of zeroes in disks centered at the origin for functions in the Fock space. This bound can be used to give a rough idea about how densely one can choose the set of zeroes of $F \in \calF^2(\bbC)$. 

\begin{lemma}
    \label{lem:roughidea}
    Let $f \in L^2(\bbR)$ be non-zero. Let $n : (0,\infty) \to \bbN_0$ be such that $n(r)$ denotes the number of zeroes of $\calB f$ (counted with multiplicity) in the disc of radius $r > 0$ centered at the origin. Let
    $m \in \bbN_0$ denote the multiplicity of the zero at the origin of $\calB f$ and let $c := 2 \log \lVert f \rVert_2 - 2 \log \lvert \lim_{z \to 0} z^{-m} \calB f(z) \rvert - m$.
    Then,
    \[
        n(r) \leq \pi \rme r^2 - 2m \log r + c, \qquad r > 0.
    \]
\end{lemma}

\begin{proof}
    Let $R \geq r > 0$ be such that $R^2 = \rme r^2$ and define 
    \[
        F(z) := \frac{\calB f(z)}{z^m}, \qquad z \in \bbC.
    \]
    Denote by $n_F(t)$ the number of zeroes of $F$ (counted with multiplicity) in the disc of radius $t  > 0$ centered at the origin. Note that $F$ is an entire function with removable singularity at $z = 0$. Therefore, Jensen's formula implies 
    \[
        \int_0^{R} \frac{n_F(t)}{t} \dd t = \frac{1}{2 \pi} \int_0^{2\pi} \log \abs{F(R \rme^{\rmi \theta})} \dd \theta - \log \abs{\lim_{z \to 0} F(z)}.
    \]
    We can use the reproducing property of $\calF^2(\bbC)$ along with the fact that $\calB : L^2(\bbR) \to \calF^2(\bbC)$ is an isometry to see that 
    \[
        \abs{F(R \rme^{\rmi \theta})} =  \frac{\abs{\calB f (R \rme^{\rmi \theta})}}{R^m}
        \leq \frac{\rme^{\frac{\pi}{2} R^2}}{R^m} \norm{\calB f}_{L^2(\bbC,\gamma)}
        = \frac{\rme^{\frac{\pi \rme}{2} r^2}}{\rme^{\frac{m}{2}} r^m} \lVert f \rVert_2.
    \]
    In addition, we can use that $n_F(t)$ is increasing in $t$ to estimate
    \[
        \int_r^R \frac{n_F(t)}{t} \dd t \geq n_F(r) \int_r^R \frac{1}{t} \dd t = \frac{n_F(r)}{2}
    \]
    and thus 
    \[
        n_F(r) \leq 2 \int_r^R \frac{n_F(t)}{t} \dd t \leq 2 \int_0^R \frac{n_F(t)}{t} \dd t.
    \]
    We conclude that 
    \[
        n_F(r) \leq \pi \rme r^2 - 2 m \log r + 2 \log \lVert f \rVert_2 - 2 \log \left\lvert \lim_{z \to 0} \frac{\calB f (z)}{z^m} \right\rvert - m.
    \]
\end{proof}

We apply Lemma~\ref{lem:roughidea} in our example: let us use the estimate \cite[Theorem~I on pp.~15--16]{levitan1987asymptotic}
\[
    n(r) = \abs{ (a \bbZ + \rmi b \bbZ) \cap \set{ z \in \bbC }{ \abs{z} \leq r } } = \frac{\pi r^2}{ab} + \calO(r^{2/3}), \qquad r \to \infty.
\]
A comparison with Lemma~\ref{lem:roughidea} yields that $ab \geq \rme^{-1}$ is a necessary condition for $F \in \calF^2(\bbC)$, where $F$ is defined in equation~\eqref{eq:construction}. We have therefore proven the following sampling result for the Gabor transform.

\begin{corollary}\label{cor:sampling}
    Let $a,b > 0$ such that $ab < \rme^{-1}$ and let $f,g \in L^2(\bbR)$. If
    \[
        \calG f = \calG g \qquad \mbox{on } a \bbZ \times b \bbZ,
    \]
    then $f = g$.
\end{corollary}

\begin{remark}
    The assumption $ab < \rme^{-1}$ is known to be suboptimal. In fact, Corollary~\ref{cor:sampling} holds under the weaker assumption $ab \leq 1$ according to \cite{lyubarskii1992frames,seip1992density}.

    Finally, we note that we can directly estimate the type of $F$ defined in equation~\eqref{eq:construction}. Indeed, $F$ is a special case of the Weierstrass $\sigma$-function and one may use its quasi-periodicity to show that its type $\sigma$ satisfies $\sigma \leq \tfrac{\pi}{2ab}$ \cite{groechenig2009gabor,hayman1974local}. Hence, $F \in \calF^2(\bbC)$ and it follow that we can construct a non-zero function $f \in L^2(\bbR)$ such that $\calG f = 0$ on $a \bbZ \times b \bbZ$ if $ab > 1$.
\end{remark}

\appendix

\section{Expansion of a function as an infinite product}

The following expansion is used in the proof of Theorem~\ref{thm:infinite}. 

\begin{proposition}\label{prop:expansion}
    For all $a,z \in \bbC$ with $a \not\in \pi \rmi \bbZ$, the following product formula holds:
    \begin{equation*}
        \frac{\sinh ( a-z )}{\sinh a} \exp( z \coth a ) = \prod_{k \in \bbZ} \left(1 - \frac{z}{a + \pi \rmi k} \right) \exp\left( \frac{z}{a + \pi \rmi k} \right).
    \end{equation*}
\end{proposition}

\begin{proof} 
    We want to rearrange the infinite product. To justify this, we show that the product converges absolutely. Let us fix $a,z \in \bbC$ with $a \not\in \pi \rmi \bbZ$ and introduce the notation $u_k := z/(a + \pi \rmi k)$ for $k \in \bbZ$. If $u_k$ satisfies $\lvert u_k \rvert \leq \tfrac12$, then 
    \begin{equation*}
        \lvert s_k \rvert := \Bigg\lvert \sum_{\ell = 2}^\infty \frac{u_k^\ell}{\ell} \Bigg\rvert \leq \frac12 \sum_{\ell = 2}^\infty \abs{u_k}^\ell = \frac{\lvert u_k \rvert^2}{2 - 2 \lvert u_k \rvert} \leq \lvert u_k \rvert^2 \leq \frac14.
    \end{equation*}
    If $\lvert k \rvert$ is sufficently large, then $\lvert u_k \rvert \leq \tfrac12$ as well as $\pi \lvert k \rvert > \lvert a \rvert$. Summing over all such $k$ yields
    \begin{align*}
        \sum_k \lvert 1 - E(u_k;1) \rvert &= \sum_k \lvert 1 - \exp( - s_k ) \rvert \leq \sum_k ( \exp \lvert s_k \rvert - 1 ) \leq 2 \sum_k \lvert s_k \rvert \\
        &\leq 2 \sum_k \lvert u_k \rvert^2 \leq 2 \lvert z \rvert^2 \sum_k (\pi \lvert k \rvert - \lvert a \rvert)^{-2} < \infty
    \end{align*}
    by the integral test.

    We rearrange the infinite product
    \begin{equation*}
        \prod_{k \in \bbZ} E(u_k; 1) = E(u_0;1) \prod_{n \in \bbN} E(u_n;1) E(u_{-n};1).
    \end{equation*}
    It is useful for us to note that 
    \begin{equation*}
        1 - u_n = \frac{1 + \frac{a-z}{\pi \rmi n}}{1 + \frac{a}{\pi \rmi n}}, \qquad (1-u_n)(1-u_{-n}) = \frac{1 + \frac{(a-z)^2}{\pi^2 n^2}}{1 + \frac{a^2}{\pi^2 n^2}},
    \end{equation*}
    which allows us to write our infinite product as 
    \begin{equation*}
        \frac{(a-z) \prod_{n \in \bbN} \big( 1 + \frac{(a-z)^2}{\pi^2 n^2} \big)}{a \prod_{n \in \bbN} \big(1 + \frac{a^2}{\pi^2 n^2}\big)} \exp\left( \frac{z}{a} + 2 a z \sum_{n \in \bbN} \frac{1}{a^2 + \pi^2 n^2}\right).
    \end{equation*}
    By the well-known expansion of the sine as an infinite product (cf.~e.g.~\cite[Section~3.23 on pp.~113--114]{titchmarsh1939theory}) and the partial fraction expansion of the cotangent (cf.~e.g.~\cite[Example~(i) on p.~113]{titchmarsh1939theory}), the above is equal to 
    \begin{equation*}
        \frac{\sin(\rmi (a-z))}{\sin(\rmi a)} \exp(z \rmi \cot(\rmi a)),
    \end{equation*}
    which proves the proposition.
\end{proof}

\section{Characterisation of finite order entire functions with common magnitudes on equidistant parallel lines}\label{app:infinite}

We prove the following characterisation.

\begin{theorem}
    Let $y_0 > 0$, and let $f,g \in \calO(\bbC) \setminus \{0\}$ be of order less than two. Then, $\abs{f} = \abs{g}$ on $\bbR + \rmi y_0 \bbZ$ if and only if $f$ and $g$ factor as 
    \begin{gather*}
        f(z) = r \rme^{\rmi \alpha} \rme^{Q(z)} P_\mathrm{c}(z) \prod_{a \in W_-} \left[ \frac{\rme^{\frac{\pi}{y_0}(a-z)}-1}{\rme^{\frac{\pi}{y_0}a}-1} \right]^{M_\mathrm{d}(a)} \prod_{a \in W_+} \left[ \frac{1 - \rme^{-\frac{\pi}{y_0}(a-z)}}{1 - \rme^{-\frac{\pi}{y_0}a}} \right]^{M_\mathrm{d}(a)}, \\
        g(z) = r \rme^{\rmi \beta} \rme^{Q(z)} P_\mathrm{c}(z) \prod_{a \in W_-} \left[ \frac{\rme^{\frac{\pi}{y_0}(\overline a-z)}-1}{\rme^{\frac{\pi}{y_0}\overline a}-1} \right]^{M_\mathrm{d}(a)} \prod_{a \in W_+} \left[ \frac{1 - \rme^{-\frac{\pi}{y_0}(\overline a-z)}}{1 - \rme^{-\frac{\pi}{y_0}\overline a}} \right]^{M_\mathrm{d}(a)}.
    \end{gather*}
    Here, $r > 0$, $\alpha,\beta \in \bbR$, $Q \in \bbC_2[z]$ is a polynomial with $Q(0) = 0$, $P_\mathrm{c}$ is the canonical product formed with the common zeroes of $f$ and $g$, and 
    \begin{gather*}
        W_+ := \set{ a \in [0,\infty) + \rmi (-y_0,y_0] }{ M_\mathrm{d}(a) > 0 }, \\
        W_- := \set{ a \in (-\infty,0) + \rmi (-y_0,y_0] }{ M_\mathrm{d}(a) > 0 }.
    \end{gather*}
\end{theorem}

\begin{proof}
    As in the proof of Theorem~\ref{thm:infinite}, we can assume that $y_0 = \tfrac{\pi}{2}$ without loss of generality. Suppose that $\lvert f \rvert = \lvert g \rvert$ on $\bbR + \tfrac{\pi \rmi}{2} \bbZ$. Then, Theorem~\ref{thm:infinite} implies that $f$ and $g$ factor as 
    \begin{gather*}
        f(z) = \rme^{Q_f(z)} P_\mathrm{c}(z) \prod_{a \in W_- \cup W_+} \left[ \frac{\sinh (a-z)}{\sinh a} \rme^{R_a(z)} \right]^{M_\mathrm{d}(a)}, \\
        g(z) = \rme^{Q_g(z)} P_\mathrm{c}(z) \prod_{a \in W_- \cup W_+} \left[ \frac{\sinh (\overline a - z)}{\sinh a} \rme^{R_{\bar a}(z)} \right]^{M_\mathrm{d}(a)}.
    \end{gather*}
    Here, $Q_f, Q_g \in \bbC_2[z]$ are polynomials, $P_\mathrm{c}$ is the canonical product formed with the common zeroes of $f$ and $g$, and $R_a(z) = z \coth a + z^2/2 \csch^2 a$.

    We want to rearrange the infinite products over $W_-$ and $W_+$ which we will justify by their absolute convergence (just as in the proof of Proposition~\ref{prop:expansion}). Let us note that almost all $a \in W_- \cup W_+$ satisfy $\lvert \Re a \rvert \geq 2\lvert z \rvert$ and focus on those that do in the following. We introduce the notation $u_{a,k} := z/(a + \pi \rmi k)$ and find that $\lvert u_{a,k} \rvert \leq \tfrac12$. Applying the well-known Taylor series expansion of $\log(1-z)$, it follows that
    \begin{equation*}
        \log \left( \frac{\sinh (a-z)}{\sinh a} \rme^{R_a(z)} \right) = \sum_{k \in \bbZ} \log E(u_{a,k};2) = - \sum_{k \in \bbZ} \sum_{\ell = 3}^\infty \frac{u_{a,k}^\ell}{\ell} =: s_a,
    \end{equation*}
    and that 
    \begin{equation*}
        \lvert s_a \rvert \leq \frac13 \sum_{k \in \bbZ} \sum_{\ell = 3}^\infty \lvert u_{a,k} \rvert^\ell \leq \frac23 \sum_{k \in \bbZ} \lvert u_{a,k} \rvert^3.
    \end{equation*}
    Therefore, 
    \begin{align*}
        \sum_{\lvert \Re a \rvert \geq 2 \lvert z \rvert} \left\lvert \frac{\sinh (a-z)}{\sinh a} \rme^{R_a(z)} - 1 \right\rvert &= \sum_{\lvert \Re a \rvert \geq 2 \lvert z \rvert} \lvert \rme^{s_a} - 1 \rvert \leq \sum_{\lvert \Re a \rvert \geq 2 \lvert z \rvert} ( \rme^{\lvert s_a \rvert} - 1 ) \\
        &\leq 2 \sum_{\lvert \Re a \rvert \geq 2 \lvert z \rvert} \lvert s_a \rvert \leq \frac43 \sum_{\lvert \Re a \rvert \geq 2 \lvert z \rvert} \sum_{k \in \bbZ} \lvert u_{a,k} \rvert^3 \\
        &= \frac{4 \lvert z \rvert}{3} \sum_{\lvert \Re a \rvert \geq 2 \lvert z \rvert} \sum_{k \in \bbZ} \lvert a + \pi \rmi k \rvert^{-3} < \infty,
    \end{align*}
    where we use that the exponent of convergence of the zeroes of $f$ is smaller than its order (cf.~e.g.~\cite[Section~8.22 on pp.~249--250]{titchmarsh1939theory}) and thus smaller than two.
    
    We have shown that the products over $W_- \cup W_+$ converge absolutely, which implies that they converge unconditionally. We can thus rearrange them to 
    \begin{multline*}
        \prod_{a \in W_-} \left[ \frac{\sinh (a-z)}{\sinh a} \rme^{R_a(z)} \right]^{M_\mathrm{d}(a)} \prod_{a \in W_+} \left[ \frac{\sinh (a-z)}{\sinh a} \rme^{R_a(z)} \right]^{M_\mathrm{d}(a)} \\
        = \prod_{a \in W_-} \left[ \frac{\rme^{2(a-z)}-1}{\rme^{2a}-1} \rme^{z(\coth a + 1) + \frac{z^2}{2} \csch^2 a} \right]^{M_\mathrm{d}(a)} \\
        \cdot \prod_{a \in W_+} \left[ \frac{1 - \rme^{-2(a-z)}}{1 - \rme^{-2a}} \rme^{z(\coth a - 1) + \frac{z^2}{2} \csch^2 a} \right]^{M_\mathrm{d}(a)}.
    \end{multline*}
    As in \cite[Section~3 on pp.~11--14]{groechenig2020phase}, we can show that the above simplifies to 
    \begin{equation*}
        \rme^{R_f(z)} \prod_{a \in W_-} \left[ \frac{\rme^{2(a-z)}-1}{\rme^{2a}-1} \right]^{M_\mathrm{d}(a)} \prod_{a \in W_+} \left[ \frac{1 - \rme^{-2(a-z)}}{1 - \rme^{-2a}} \right]^{M_\mathrm{d}(a)},
    \end{equation*}
    where $R_f \in \bbC_2[z]$, because all of the sums
    \begin{gather*}
        \sum_{a \in W_-} \abs{\frac{\rme^{2(a-z)}-1}{\rme^{2a}-1} - 1}, \qquad \sum_{a \in W_+} \abs{\frac{1 - \rme^{-2(a-z)}}{1 - \rme^{-2a}} - 1}, \\
        \sum_{a \in W_-} \abs{\coth a + 1}, \qquad \sum_{a \in W_+} \abs{\coth a - 1}, \qquad \sum_{a \in W_- \cup W_+} \lvert \csch a \rvert^2
    \end{gather*}
    converge. We conclude that 
    \begin{gather*}
        f(z) = \rme^{Q_f(z)+R_f(z)} P_\mathrm{c}(z) \prod_{a \in W_-} \left[ \frac{\rme^{2(a-z)}-1}{\rme^{2a}-1} \right]^{M_\mathrm{d}(a)} \prod_{a \in W_+} \left[ \frac{1 - \rme^{-2(a-z)}}{1 - \rme^{-2a}} \right]^{M_\mathrm{d}(a)}, \\
        g(z) = \rme^{Q_g(z)+R_g(z)} P_\mathrm{c}(z) \prod_{a \in W_-} \left[ \frac{\rme^{2(\overline a-z)}-1}{\rme^{2\overline a}-1} \right]^{M_\mathrm{d}(a)} \prod_{a \in W_+} \left[ \frac{1 - \rme^{-2(\overline a-z)}}{1 - \rme^{-2\overline a}} \right]^{M_\mathrm{d}(a)}.
    \end{gather*}

    It remains for us to consider the polynomials $p_f := Q_f+R_f, p_g := Q_g+R_g \in \bbC_2[z]$: since $\lvert f \rvert = \lvert g \rvert$ on $\bbR$, we have $\Re p_f = \Re p_g$ on $\bbR$, which implies that the coefficients of $p_f$ and $p_g$ have the same real parts. Finally, $\lvert f \rvert = \lvert g \rvert$ on $\bbR + \tfrac{\pi \rmi}{2}$ implies that $p_f = p_g + \rmi c$ for $c \in \bbR$ as in the proof of Theorem~\ref{thm:three_lines}. Therefore, the factorisation has been proven. The reverse implication follows by a direct computation.
\end{proof}

\paragraph{Acknowledgements} I want to thank Rima Alaifari and Giovanni S.~Alberti for fruitful discussions and I want to thank the reviewer for their helpful inputs which have greatly improved the presentation of this paper. I would also like to acknowledge funding through the SNSF Grant 200021\_184698 and that I used ChatGPT 3.5 during the revision process to enhance the readability and language of certain sentences. I confirm that I have diligently reviewed and edited the output, recognising my responsibility and accountability for the contents of this paper.

Finally, I would like to acknowledge that Theorem~\ref{thm:three_lines} has been suggested to me twice: first, by Philippe Jaming in the summer of $2022$ and, secondly, by the reviewer during the review process. I want to thank them for this; all credit for Theorem~\ref{thm:three_lines} should go to them. I was also informed that Theorem~\ref{thm:three_lines} follows from a result in \cite{liehr2023arithmetic}.

\paragraph{Declarations of interest} None.

\bibliography{sources}
\bibliographystyle{plain}

\end{document}